\documentclass[a4paper,11pt]{amsart}
\usepackage{amsfonts}
\usepackage{amssymb}
\usepackage[utf8]{inputenc}
\usepackage{amsmath}
\usepackage{pdflscape}
\usepackage{float}
\usepackage{easyReview}
\usepackage{graphicx}
\usepackage[colorlinks=true, linkcolor=red, citecolor=blue]{hyperref}
\usepackage[]{epsfig}
\usepackage[]{pstricks}
\usepackage{tikz}
\newcommand{\nwc}{\newcommand}
\nwc\eps{\varepsilon}
\usepackage{caption}
\usepackage{subcaption}

\newcommand\nd{\noindent}

\newcommand\I{{\mathcal{I}}}

\newcommand\R{\mathbb{R}}

\newcommand\N{\mathbb{N}}

\nwc\EN{{\mathcal{E}_1}}

\newcommand\ep{\epsilon}

\nwc{\dx}{\partial_x}
\nwc{\dy}{\partial_y}

\newcommand\M{{\mathcal{M}}}

\nwc{\hamone}{{\mathcal{H}}}
\nwc{\Imu}{\Sigma}
\nwc{\Iem}{\I}
\nwc{\Gep}{G}
\nwc{\wt}{\widetilde}
\nwc{\IF}{\mathcal F}

 \newtheorem{thm}{Theorem}[section]
 
 \newtheorem{lem}[thm]{Lemma}
 
 \theoremstyle{definition}
 
 \theoremstyle{remark}
 \newtheorem{rem}[thm]{Remark}
 
 \numberwithin{equation}{section}
\begin{document}

\title[Traveling-wave solutions for a Boussinesq system]{
 Traveling-wave solutions for a higher-order Boussinesq system: existence and numerical analysis}

%\thanks{This work was completed with the support of our
%----------Author 2

\author[Capistrano-Filho]{Roberto de A. Capistrano--Filho}

\address{%Roberto de A. Capistrano--Filho, 
Department of Mathematics, Universidade Federal de Pernambuco\\
Av. Prof. Moraes Rego, 1235 - Cidade Universitária, Recife - PE, 50670-901, Brasil}

\email{roberto.capistranofilho@ufpe.br}

\author[Mu\~noz]{Juan Carlos Mu\~noz*}

\address{
%Juan Carlos Mu\~noz, 
Department of Mathematics, Universidad del Valle\\
Calle 13, 100-00\\
Cali - Colombia}

\email{juan.munoz@correounivalle.edu.co}

%----------Author 1
\author[Quintero]{Jos\'e R. Quintero}

\address{%Jos\'e R. Quintero, 
Department of Mathematics, Universidad del Valle\\
Calle 13, 100-00\\
Cali - Colombia}

\email{jose.quintero@correounivalle.edu.co}


\date{\today}
\thanks{*Corresponding author: juan.munoz@correounivalle.edu.co}

\begin{abstract}
We study the existence and numerical computation of traveling wave solutions for a family of nonlinear higher-order Boussinesq evolution systems with a Hamiltonian structure. This general Boussinesq evolution system includes a broad class of homogeneous and non-homogeneous nonlinearities. We establish the existence of traveling wave solutions using the variational structure of the system and the \textit{concentration-compactness} principle by P.-L. Lions, even though the nonlinearity could be non-homogeneous.  For the homogeneous case, the traveling wave equations of the Boussinesq system are approximated using a spectral approach based on a Fourier basis, along with an iterative method that includes appropriate stabilizing factors to ensure convergence. In the non-homogeneous case, we apply a collocation Fourier method supplemented by Newton's iteration. Additionally, we present numerical experiments that explore cases in which the wave velocity falls outside the theoretical range of existence.
\end{abstract}

%%% ----------------------------------------------------------------------
\maketitle
%\tableofcontents
%%% ----------------------------------------------------------------------

\section{Introduction}
Bona, Chen, and Saut, in \cite{BCS02} and \cite{BCS04}, derived some classical Boussinesq systems corresponding to the first- and second-order approximations to the full two-dimensional Euler equations to describe the motion of short waves of small amplitude on the surface of an ideal fluid under gravity force. In particular, the authors derived a four-parameter family of Boussinesq systems from the two-dimensional Euler equations, known as the $abcd$ Boussinesq system, 
\begin{equation}
\begin{cases}
\begin{array}{rl}
\left(I- b\partial^2_x\right) \partial_t\eta  + \partial_x u
+a\partial_x^3 u
&= -\partial_x\left(\eta u \right),
\\ \label{abcd}
\left(I-d\partial_x^2\right) \partial_t u
+\partial_x \eta+c\partial_x^3 \eta
&=\frac12\partial_x\left(u^2 \right),
\end{array}
\end{cases}
\end{equation}
with $a+b+c+d=\frac13-\sigma$ ($\sigma\geq 0$ is the surface tension). It is known that system \eqref{abcd} admits (big) solitary-wave solutions in certain regimes of the parameters involved in the system (for instance, see \cite{BCL-15} and the references therein for details). Moreover, when $b=d>0$, it was also shown in \cite{BCS02} that the system \eqref{abcd} is Hamiltonian and globally well posed in the energy space $X=H^1(\mathbb{R}) \times H^1(\mathbb{R})$, at least for small data, due to the conservation of energy.

It is important to mention that in the same articles \cite{BCS02,BCS04} the authors introduced an eight-parameter family of Boussinesq systems, namely
\begin{equation} 
\begin{cases}
\begin{array}{rl}
\left(I- d\partial^2_x+d_2\partial_x^4\right) \partial_t u  + \partial_x\eta
+c\partial_x^3\eta+c_2\partial_x^5 \eta
= \partial_x\left(H_1(\eta, \partial_x\eta, \partial^2_x\eta, u, \partial_x u, \partial^2_x u) \right),
\\ \label{8-BS} \\
\left(I-b\partial_x^2+b_2\partial_x^4\right) \partial_t\eta
+\partial_x u+a\partial_x^3 u +a_2\partial_x^5 u
=\partial_x\left(H_2(\eta, \partial_x\eta, \partial^2_x\eta, u, \partial_x u,\partial^2_x u) \right),
\end{array}
\end{cases}
\end{equation}
where $H_i$ for $i=1,2$ are given by
\begin{align*}
H_1(\eta,\partial_x \eta,\partial^2_x \eta, u,\partial_x u,\partial^2_x u) &=  \frac12 u^2 -c\partial_x(u\partial_x u)-\eta \partial^2_x\eta-\frac12 \partial_x(u^2) +(c+d)u \partial^2_x u,
\end{align*}
and
\begin{align*}
H_2(\eta, \partial_x\eta,\partial^2_x \eta, u,\partial_x u,\partial^2_x u) &=  -\eta u + b\partial^2_x(\eta u)-\left(a+b-\frac13\right)\eta \partial^2_x u,
\end{align*}
with the constants $a$,  $a_2$, $b$, $b_2$, $c_2$, $d$, and $d_2$ satisfying the following condition:
\[
a+b+c+d=\frac13-\sigma,
\]
and also that 
 \[
a_2 -b_2=-\frac12 (\theta^2-\frac13)b+\frac5{24}(\theta^2-\frac15)^2,  \ \
c_2 -d_2=-\frac12 (1-\theta^2)c+\frac5{24}(1-\theta^2)(\theta^2-\frac15).
\]
Here, $\theta\in [0, 1]$ is a fixed modeling parameter, which does not possess a direct physical interpretation, as mentioned in \cite{BCS02,BCS04}.

In the present work, we study the existence of traveling waves of finite energy for the one-dimensional higher-order Boussinesq evolution system derived from \eqref{8-BS}, namely
\begin{equation} \left\{
\begin{array}{rl}
\left(I- d\partial^2_x+d_2\partial_x^4\right) \partial_tu  + \partial_x\eta+c\partial_x^3\eta+c_2\partial_x^5 \eta
&= \partial_x\left(H_1(\eta, \partial_x\eta, \partial_x^2\eta, u,\partial_x u, \partial^2_xu) \right),
\\ \label{1bbl} \\
\left(I-b\partial_x^2+b_2\partial_x^4\right) \partial_t\eta
+\partial_x u+a\partial_x^3 u +a_2\partial_x^5 u
&=\partial_x\left(H_2(\eta, \partial_x\eta, \partial_x^2\eta, u,\partial_x u, \partial^2_xu) \right),
\end{array}\right.
\end{equation}
where $\eta=\eta(x,t)$ and $u=u(x,t)$ are real-valued
functions, and the nonlinearity $H=(H_1, H_2)^t$ has the variational structure
\begin{equation*}
\begin{split}
H_1(q, r, z, s, t, w)= &F_q(q, r, s, t)- rF_{qr}(q, r, s, t)- zF_{rr}(q, r, s, t)- tF_{sr}(q, r, s, t)\\& - wF_{tr}(q, r, s, t),
\end{split}
\end{equation*}
and 
\begin{equation*}
\begin{split}
H_2(q, r, z, s, t, w)=& F_s(q, r, s, t)- rF_{qt}(q, r, s, t)- zF_{rt}(q, r, s, t)- tF_{st}(q, r, s, t)\\&- wF_{tt}(q, r, s, t),
\end{split}
\end{equation*}
where $F$  is a function with some properties, such as being $p+2$-homogeneous (see Section \ref{sec2} for some examples of $F$). Let us briefly outline some aspects of the system considered in this work.

\subsection{Background} It is important to highlight that, similar to the case of the $abcd$-Boussinesq system and the KdV equation at the traveling wave level (see, for example, \cite{CQS-2024}), the higher-order Boussinesq system \eqref{1bbl} at the traveling wave level is, depending on the nonlinearity, related to traveling wave solutions of the fifth-order KdV equation:
\begin{equation}\label{5kw}
\partial_tu + \alpha\partial^3_x u + \beta \partial^5_xu = \partial_x(f(u, \partial_xu, \partial^2_xu)),
\end{equation}
where the nonlinearity takes the variational form
$$f(q, r, s) = F_q(q, r) - rF{qr}(q, r) - sF_{rr}(q, r),$$
for some $C^2$ function $F$, which is not necessarily homogeneous, as discussed by Esfahani and Levandosky in \cite{Esfahani-Levandosky-2021}. We mention that the model \eqref{5kw} has applications in various physical phenomena, as we see in different works. See, for instance, the references \cite{Craig-Groves-1994,Hunter-Scheurle-1988,Kawahara-1972,Olver-1984,Zufiria-1987} and therein.

%Numerical methods for approximating the solutions of initial-boundary value problems associated with Boussinesq-type systems have been extensively studied. For instance, in \cite{Dougalis}, the standard Galerkin finite element method was used for spatial discretization, combined with a fourth-order explicit Runge-Kutta scheme for time integration. In \cite{Dursun}, soliton propagation and interactions were investigated using two finite difference schemes and two finite element methods with second and third-order time discretizations. In \cite{Ersoy}, a collocation method based on exponential cubic B-spline functions was proposed to solve one-dimensional Boussinesq systems and simulate the motion of traveling waves. Furthermore, in \cite{Duran}, a Fourier collocation method was applied to study the dynamics of solitary-wave solutions under periodic boundary conditions. Finally, for more details on higher-order Boussinesq systems, we refer the reader to the works of Benney \cite{Benney-1977}, Craig and Groves \cite{Craig-Groves-1994}, Hunter and Scheurle \cite{Hunter-Scheurle-1988}, Kichenassamy and Olver \cite{Kichenassamy-Olver-1991}, Olver \cite{Olver-1984}, Ponce \cite{Ponce-1993}, Bona, Chen, and Saut \cite{BCS02,BCS04}, as well as Chen, Nguyen, and Sun \cite{CNS2011}, and Bona, Colin, and Lannes \cite{BCL2005}.

Numerical methods for approximating the solutions of initial–boundary-value problems associated with Boussinesq-type systems have been extensively investigated. For example, in \cite{Dougalis}, the standard Galerkin finite element method was used for spatial discretization, combined with a fourth-order explicit Runge–Kutta scheme for time integration. In \cite{Dursun}, soliton propagation and interactions were studied using two finite difference schemes and two finite element methods with second and third-order time discretizations. In \cite{Ersoy}, a collocation method based on exponential cubic B-spline functions was proposed to solve one-dimensional Boussinesq systems and simulate the motion of traveling waves. Furthermore, in \cite{Duran}, a Fourier collocation method was applied to study the dynamics of solitary-wave solutions under periodic boundary conditions.

In addition to these numerical approaches, we also highlight the recent works of Klein and Saut \cite{KS-24,KS-25}, which provide detailed numerical investigations of the Amick–Schonbek and Kaup–Broer–Kupershmidt systems. These studies survey and numerically analyze notable members of the so-called (abcd) family of Boussinesq systems.

Finally, for a more general perspective on higher-order Boussinesq systems, we refer the reader to the works of Benney \cite{Benney-1977}, Craig and Groves \cite{Craig-Groves-1994}, Hunter and Scheurle \cite{Hunter-Scheurle-1988}, Kichenassamy and Olver \cite{Kichenassamy-Olver-1991}, Olver \cite{Olver-1984}, Ponce \cite{Ponce-1993}, Bona, Chen, and Saut \cite{BCS02,BCS04}, as well as Chen, Nguyen, and Sun \cite{CNS2011}, and Bona, Colin, and Lannes \cite{BCL2005}.

We mention that this is only a brief overview of the state of the art. Therefore, we encourage the reader to refer to the references cited in the mentioned works for further details.

\subsection{Notations and main results} With this overview of the state of the art, we now present the primary objective of this work: proving the existence of traveling-wave solutions for the generalized Boussinesq system \eqref{1bbl}. To the best of our knowledge, this is the first result addressing this topic for higher-order Boussinesq systems. Specifically, our objective is to demonstrate the existence of traveling-wave solutions for the system \eqref{1bbl}. In other words, we seek solutions $(\eta, u)$ of the form:
\begin{equation*}\label{etaPhi}
\eta(t,x)=  \psi\left({x-\omega
t}\right), \ \ u(t,x)=
v\left({x-\omega t}\right).
\end{equation*}
In this case, we find that the traveling wave profile $(\psi, v)$ should
satisfy the system
\begin{equation}
\left\{
\begin{array}{rl}
-\omega\left(v-dv''+d_2 v^{(iv)}\right)+ \psi +c\psi''+c_2 \psi^{(iv)}
-H_1(\psi, \psi', \psi'', v,  v', v'')&=0,
\\ \label{trav-eqs}\\
-\omega\left(\psi-b\psi''+b_2 \psi^{(iv)}\right) + v+a v'' +a_2 v^{(iv)}
-H_2(\psi, \psi', \psi'', v,  v', v'')& =0,
\end{array}\right.
\end{equation}
where we are assuming that $H_1(0,0,0,0,0,0)=H_2(0,0,0,0,0,0) = 0$, $b=d>0$, $b_2=d_2>0$, $a, c<0$ and that $ a_2, c_2>0$. Furthermore, we point out that, since solitary waves decay to zero at infinity, the constant terms that would arise in system \eqref{trav-eqs} as a result of integrating equations \eqref{1bbl} must necessarily vanish.

Note that the existence of solitons, traveling waves with finite energy, for the higher-order Boussinesq system \eqref{1bbl} follows from a variational approach using a minimax-type result. Specifically, solutions $(\psi, v)$ of the system \eqref{trav-eqs} are critical points of the functional $J_{\omega}$ given by:
\begin{equation}\label{J_w}
J_{\omega}(\psi, v)=\frac12 I_{\omega}(\psi, v)- K(\psi, v),
\end{equation}
where the functionals $I_\omega$ and $K$ are defined on the space $
X= H^2(\R)\times H^2(\R)$ by
\begin{equation}\label{I_w}
I_{\omega}(\psi,v)=I_1(\psi,v)+I_{2,{\omega}}(\psi,v),
\end{equation}
and
\begin{equation}\label{K-psi-v}
K(\psi, v)  =   \int_{\mathbb R}F(\psi, \psi', v,  v')\,dx, 
\end{equation}
with 
\begin{equation}\label{I1a}
I_1(\psi, v)
 =  \int_{\mathbb R}\left[\psi^2 -c(\psi')^{2} + c_2(\psi'')^{2}+
v^2-a( v')^2+a_2( v'')^2\right]dx,
\end{equation}
and
\begin{equation}\label{I2a}
I_{2}(\psi, v)
=  \int_{\mathbb R}\left(\psi-b\psi''+b_2 \psi^{(iv)} \right)v\,dx= \int_{\mathbb R}\left(\psi v+b\psi' v'+b_2 v''\psi'' \right)\,dx,
\end{equation}
where $I_{2, \omega}=-2\omega I_{2}$. Before we go further, from the assumptions we observe
\[
K'(\psi, v)= (H_1(\psi, \psi', \psi'', v,  v', v''),
H_2(\psi, \psi', \psi'', v,  v', v''))^t.
\]
In the following, we will say that weak solutions for (\ref{trav-eqs}) are critical points of the functional $J_{\omega}$. Now, we set up
\begin{equation}\label{P_w}
P_{\omega}(\psi, v) =\left<J'_{\omega}(\psi, v), (\psi, v)\right> = I_\omega(\psi, v)-N(\psi, v), 
\end{equation}
and
\begin{equation}\label{N-pv}
N(\psi, v) =\left<K'(\psi, v), (\psi, v)\right>=\int_{\R}(u, u_x, v, v_x)\cdot \nabla F(u, u_x, v, v_x)\,dx.
\end{equation}
Remark that the quantity \eqref{P_w} is important for the analysis of the orbital stability of ground-state solutions as proposed in \cite[Theorems 2 and 3]{GSS}. See, for instance \cite{CQS-2024,HTS-13}, and the references therein for more details on the orbital and spectral stability problem related to the system \eqref{abcd}.
 
From this, we consider the manifold $\mathcal M_{\omega}\subset H^2(\R)\times H^2(\R)$ given by 
\begin{equation}\label{M_w}
\mathcal M_{\omega}=\{(\psi, v)\in H^2(\R)\times H^2(\R): \ P_{\omega}(\psi, v)=0 \},
\end{equation}
and define the number
\begin{equation}\label{sw}
S(\omega)=\inf\{J_\omega(\psi, v): (\psi, v)\in \M_{\omega}\}.
\end{equation}
Note that $\mathcal{M}_{\omega}$ is just the ``artificial constraint" for minimizing the functional $I_\omega$ on $X$.

Now, we state the assumptions on the nonlinear part $F$ in terms of the function $$G(w)= w\cdot \nabla F(w), \quad\text{for}\quad w\in \R^4.$$ As done by Esfahani-Levandosky in \cite{Esfahani-Levandosky-2021}, taking $F$ a $C^2$ function satisfying $F(0,0)=0$, and we will consider:

\nd (a) There exists  $p>0$ such that $w\cdot \nabla G(w)\geq (p+2) G(w)$ for $w\in \R^4$.

\nd (b) There exist $0<q_1\leq q_2 <\infty$ and $C>0$ such that for $w\in \R^4$,
\[
|D^2F(w)|\leq C(|w|^{q_1}+ |w|^{q_2}).
\]
\nd (c)  There exist $u, v\in H^2(\R)$ such that
\[
\int_{\R}F(u, \partial_xu, v, \partial_xv)\,dx>0.
\]
We point out that condition (c) ensures that the functionals $K$ and $N$ are positive somewhere, which is crucial for the minimization problem (\ref{sw}) to make sense.

With these previous notations in hand, traveling-wave solutions are established using the \textit{concentration}-\textit{compactness principle} by P.-L. Lions \cite{Lions-1984a, Lions-1984b}. So, the main result of our work can be read as follows. 
\begin{thm} \label{exist}(Existence of traveling waves). Let  $0<|\omega|<\min\left\{1,\frac{{-a}}{b}, \frac{{-c}}{b},\frac{a_2}{b_2},\frac{c_2}{b_2}\right\}$ and the function $F$ satisfying items (a), (b) and (c). Given a  minimizing sequence  $(\psi _{n}, v_n)_n$ of $S(\omega)$,  there exist a subsequence $\left(\psi_{n_{k}}, v_{n_{k}}\right)_k$, a sequence of points $y_{k}\in \mathbb{R}$  and $(\psi_0, v_0)\in\mathcal{M}_{\omega}\setminus\{0\}$ such that 
$$
(\psi_{n_{k}}(.+y_{k}), v_{n_k}(.+y_{k}))\rightarrow (\psi_0, v_0) \quad\text{in}\quad X, \ \mbox{as $k\to \infty$,}
$$ 
and $J_{\omega}(\psi_0, v_0)=S(\omega)$. In other words, $(\psi_0, v_0)$, weak solution of \eqref{trav-eqs}, is a minimizer of $J_\omega$.
\end{thm}

The second part of this work investigates the numerical analysis of traveling-wave solutions for the generalized Boussinesq system \eqref{1bbl}, a topic not previously addressed in the literature. Spectral numerical solvers are introduced to approximate these steady solutions for a given wave velocity and model parameters, considering both homogeneous and non-homogeneous nonlinearities  $H_1$ and $H_2$. In doing so, we validate the theoretical results and successfully compute some of these solutions within the parameter regimes and velocity ranges predicted by the theory developed in this paper. Furthermore, we present numerical experiments that go beyond the theoretical framework, investigating cases where the wave velocity falls outside the predicted range of existence. Our findings suggest that the existence of traveling-wave solutions to the system \eqref{1bbl} is influenced by additional parameters, such as the power of the nonlinearities.

\subsection{Outline} This paper is organized as follows. In Section \ref{sec2}, we present the preliminaries concerning the existence of solitons (traveling wave solutions of finite energy) for the higher-order Boussinesq system, employing the \textit{concentration-compactness principle}, which fully characterizes the convergence of positive measures and provides a compact local embedding result. In Section \ref{sec3}, we introduce iterative numerical schemes called  \textit{Fourier series-based method} to approximate the solitons of the Boussinesq system \eqref{1bbl}, giving a numerical verification of our main result. 

\section{Existence of solitary waves for $b=d$ and $b_2=d_2$}\label{sec2}

In this section, we prove the existence of traveling waves for the generalized Boussinesq system \eqref{1bbl}. Based on the assumptions of Esfahani and Levandosky \cite[Lemmas 2.5 and 2.6]{Esfahani-Levandosky-2021}, we can directly extend these lemmas as follows:
\begin{lem}\label{esf-lev-2021}
Under the assumptions on $F$, for $w, \tilde w\in \R^2$ we have,

\nd i) \ \ $G(w, \tilde w)\geq (p+2)F(w, \tilde w )$.

\nd ii) \ $R(\alpha (w, \tilde w))\geq \alpha^{p+2}R(w, \tilde w)$ with $\alpha\geq 1$ and $R(\alpha(w, \tilde w))\leq \alpha^{p+2}R(w, \tilde w)$ with $0<\alpha\leq 1$ for $R=G$ or $R=F$.

\nd iii) $|F(w, \tilde w)|+|G(w, \tilde w)|+|\left<(w, \tilde w), G(w, \tilde w)\right>|\leq C(|(w,\tilde w)|^{q_1+2}+ |(w, \tilde w)|^{q_2+2})$.

\nd iv) \ $N(\alpha w)\geq \alpha^{p+2}N(w)$ with $\alpha\geq 1$ and $N(\alpha w)\leq \alpha^{p+2}N(w)$ with $0<\alpha\leq 1$.

\nd v) \ \ $N(w)\geq (p+2)K(w)$.

\nd vi) \ $\left<N'(w), w \right>\geq (p+2)N(w)$.

\nd vii) $|N(w)|+|K(w)|+|\left<N'(w), w \right>|\leq C\left(||w||_{H^2}^{q_1+2}+ ||w||_{H^2}^{q_2+2}\right)$, for $w\in H^2(\R)\times H^2(\R)$.
\end{lem}
We note that under the assumption of $(p+2)$-homogeneity on $F$, we already have that
\[
G(q, r, s, t)=qF_q(q, r, s, t)+ rF_q(q, r, s, t)+sF_s(q, r, s, t)+ tF_t(q, r, s, t)=(p+2)F(q, r, s, t).
\]
Also, taking into account Esfahani-Levandosky's work, we have the following examples of functional $F:\R^4 \to \R$ that satisfy the previous assumptions, namely
\begin{itemize} 
\item[a)] $F(u, v)=F_1(u)+F_1(v)$, where $$F_1(u)=\frac{|u|^{q+2}}{
q+2} + \frac{|u|^{p+2}}{p+2},$$ with $0<p<q$;
 \item[b)] $F(u, v)=F_1(u)+F_1(v)$, where $F_1$ is defined as
\[
F_1
(u)=\sum_j^{m} a_j|u|^{q_j+2} - \sum_k^{n} b_k|u|^{p_k+2},
\]
with $a_j, b_k>0$ and $0<p_k<q_j$ for all $j, k\in \N$;
 \item[c)]  $F(u, \partial_xu, v, \partial_xv)=F_1(u, \partial_xu)+F_1(v, \partial_xv)$, where $F_1=F_2+F_3$ with $F_2$ being a $(p+1)$-homogeneous function and $F_3$ being a $(q+1)$-homogeneous function for $0<p<q$.
 \end{itemize}
 
The next remark will be useful for the proof of the next lemma.
\begin{rem}\label{rmk22}
For $0<|\omega|<\min\left\{1,\frac{{-a}}{b}, \frac{{-c}}{b},\frac{a_2}{b_2},\frac{c_2}{b_2}\right\}$, we note that,
$$
\begin{gathered}
I_\omega(\psi, v)=\int_{\mathbb{R}}\left[(\psi-\omega v)^2+\left(\sqrt{|c|} \psi^{\prime}-\frac{b \omega}{\sqrt{|c|}} v^{\prime}\right)^2+\left(1-|\omega|^2\right) v^2+\left(|a|-\frac{b^2 \omega^2}{|c|}\right)\left(v^{\prime}\right)^2\right. \\
\left.+c_2\left(\psi''- \frac{\omega b_2}{c_2} v'' \right)^2+ \left(a_2-\frac{\omega^2 b_2^2}{c_2}\right)(v'')^2\right]dx \geq 0
\end{gathered}
$$
and
$$
\begin{gathered}
I_\omega(\psi, v)=\int_{\mathbb{R}}\left[(v-\omega \psi)^2+\left(\sqrt{|c|} v^{\prime}-\frac{b \omega}{\sqrt{|c|}} \psi^{\prime}\right)^2+\left(1-|\omega|^2\right) \psi^2+\left(|a|-\frac{b^2 \omega^2}{|c|}\right)\left(\psi^{\prime}\right)^2\right. \\
\left.+a_2\left(v''- \frac{\omega b_2}{a_2} \psi'' \right)^2+ \left(c_2-\frac{\omega^2 b_2^2}{a_2}\right)(\psi'')^2\right]dx \geq 0
\end{gathered}
$$
where we are using that $|\omega| < \frac{\sqrt{ac}}{b}$ and $|\omega| < \frac{\sqrt{a_2c_2}}{b_2}$.
\end{rem}

 With this in hand, our first lemma ensures some boundedness for the quantity \eqref{I_w}.
\begin{lem}\label{inf}
For $0<|\omega|<\min\left\{1,\frac{{-a}}{b}, \frac{{-c}}{b},\frac{a_2}{b_2},\frac{c_2}{b_2}\right\}$, we have that

\nd i) the functional $I_{\omega}$ is nonnegative and there are positive constants
$M_1(\omega, a, a_2, b, b_2, c, c_2 )$ and $M_2(\omega, a, a_2, b, b_2, c, c_2 )$ such that
\begin{equation}\label{IG1}
M_1 \|(\psi, v)\|_{X}^{2} \leq I_\omega(\psi, v)\leq M_2
 \|(\psi, v)\|_{X}^{2}.
\end{equation}
\nd ii)  $S(\omega)$ exists, is positive, and
\begin{equation}\label{swn}
S(\omega) = \inf \left\{ J_{\omega, p}(\Psi) \; ; \; \Psi \neq 0, \ P_{\omega}(\Psi) \le 0 \right\},
\end{equation}
where $J_{\omega,p}$ is defined by 
\[
J_{\omega, p}(\Psi) = J_{\omega}(\Psi) - \frac{1}{p+2} P_{\omega}(\Psi).
\]
Here $J_{\omega}$ is defined by \eqref{J_w} and $P_{\omega}$ is given by \eqref{P_w}.
\end{lem}
\begin{proof} i)  Using the definition of $I_\omega$ given by \eqref{I_w}, we directly see that 
\begin{align*}\label{eqIN}
I_\omega(\psi,v)\leq&\int_{\mathbb R}
\left[(1+|\omega|)\psi^2+\left(|c|+b|\omega| \right)(\psi')^2+(1+|\omega|)v^2+\left(|a|+b|\omega|
\right)(v')^2\right. \\
&\left.+\left(c_2+b_2|\omega|
\right)(\psi'')^2+\left(a_2+b_2|\omega|
\right)(v'')^2
\right]dx \nonumber\\
\leq& \max\left(1+|\omega|,|c|+b|\omega| , |a|+b|\omega|, a_2+b_2|\omega|, c_2+b_2|\omega|)
\right)\|(\psi,v)\|^2_{X}.
\end{align*}
On the other hand, from Young's inequality, we also have
\begin{align*}
I_\omega(\psi,v)\geq&\int_{\mathbb R}
\left[(1-|\omega|)\psi^2+\left(-c-b|\omega| \right)(\psi')^2+(1-|\omega|)v^2+\left(-a-b|\omega|
\right)(v')^2\right. \\
&\left.+\left(c_2-b_2|\omega|
\right)(\psi'')^2+\left(a_2-b_2|\omega|
\right)(v'')^2
\right]dx \nonumber\\
 \geq& \min\left(1-|\omega|,-c-b|\omega|, -a-b|\omega|, a_2-b_2|\omega|, c_2-b_2|\omega|)
\right)\|(\psi,v)\|^2_{X},
\end{align*}
showing that the inequality (\ref{IG1}) holds.

\nd ii) The proof of this item is analogous to the proof of \cite[Lemma 2.8]{Esfahani-Levandosky-2021}.
\end{proof}
\begin{rem}
 Characterization of $S(\omega)$ in the homogeneous case is given by
\[
S(\omega)=\inf \left\{ J_{\omega}(\Psi) \; ; \; \Psi \neq 0, \ P_{\omega}(\Psi) \le 0 \right\},
\]
since we have
\[
N(\Psi)=(p+2)K(\Psi), 
\]
which implies that 
\[
J_{\omega}(\Psi)= \frac{p}{2(p+2)}I_{\omega}(\Psi).
\]
In this case, we already have that 
\[
S(\omega)= \inf\left\{ \frac{p}{2(p+2)}I_{\omega}(\Psi); \Psi \ne 0, \ P_{\omega}(\Psi)\leq 0\right\}.
\]
\end{rem}
In the next result, we establish that the minimizers for $S(\omega)$ correspond to weak solutions for the higher-order Boussinesq system.
\begin{thm}
Let $(\psi_0, v_0)\in X$ be such that $P_{\omega}(\psi_0, v_0)=0$ and $J_{\omega}(\psi_0, v_0)=S(\omega)$. Then $(\psi_0, v_0)$ is a weak solution of (\ref{trav-eqs}).
\end{thm}
\begin{proof} From the Lagrange multiplier theorem, there exists $\lambda\in \R$ such that $$J'_{\omega}(\psi_0, v_0)=\lambda P'_{\omega}(\psi_0, v_0).$$ 
On the other hand,
\begin{align*}
 \left<P'_{\omega}(\psi_0, v_0), (\psi_0, v_0)\right>&= \left<I'_{\omega}(\psi_0, v_0), (\psi_0, v_0)\right>-\left<N'(\psi_0, v_0), (\psi_0, v_0)\right>
 \\&=2 I_{\omega}(\psi_0, v_0)-\left<N'(\psi_0, v_0), (\psi_0, v_0)\right>,\\
&\leq 2I_{\omega}(\psi_0, v_0)-(p+2)N(\psi_0, v_0),\\
&\leq -p I_{\omega}(\psi_0, v_0)<0,
\end{align*}
where in the first inequality we have used Lemma \ref{esf-lev-2021} item (vi) and in the third inequality that $I_{\omega}(\psi, v)=N(\psi, v)$. Thus, we must have $\lambda=0$, since $ P_{\omega}(\psi_0, v_0)=\left<J'_{\omega}(\psi_0, v_0), (\psi_0, v_0)\right>=0$, completing the proof.
\end{proof}

Before going further, let us give an important property of the minimizing sequences for $S(\omega)$.
\begin{lem}\label{lem-r} If $((\psi_n, v_n))_n$ is any minimizing sequences for $S(\omega)$, then  $(I_{\omega}(\psi_n, v_n))_n$ is a bounded sequence. Moreover, there is $L>0$ such that  $I_{\omega}(\psi_n, v_n) \to L$, as $n \to \infty$, up to a subsequence.
 \end{lem}
\begin{proof} Let $(\psi_n, v_n)\in X$ be such that $I_{\omega}(\psi_n,v_n)=N(\psi_n,v_n)$ and that $J_{\omega}(\psi_n, v_n)=S(\omega) +o(1)$, as $n\to \infty$. Now, we note that
\begin{align*}
\frac12 I_{\omega}(\psi_n, v_n)&= J_{\omega}(\psi_n, v_n)+K(\psi_n, v_n) \\
&\leq J_{\omega}(\psi_n, v_n)+\frac1{p+2}N(\psi_n, v_n)  \\
&\leq S(\omega)+\frac1{p+2}I_{\omega}(\psi_n, v_n)+o(1),
\end{align*}
which implies that $||(\psi_n, v_n)||_X$ is bounded, since
\[
\left( \frac{p}{2(p+2)}\right)I_{\omega}(\psi_n, v_n)\leq S(\omega)+o(1).
\]
So, consequently, we may assume that the sequence (or a subsequence of) $\{I_{\omega}(\psi_n, v_n)\}_n$ converges to some positive number $L$, showing the lemma. 
\end{proof}

Classically, to establish the existence of traveling-wave solutions, we will use the \textit{Lion's concentration-compactness principle} \cite{Lions-1984a,Lions-1984b}. We will apply this principle to the measure defined by the density $\rho$ given by
\begin{multline}\label{RHO}
\rho(\psi,v) =\psi^2 - c(\psi')^{2} + c_2(\psi'')^{2}+
v^2-a( v')^2+a_2( v'')^2 -2\omega\left(\psi v+b\psi' v'+b_2 v''\psi'' \right).
\end{multline}
We note from the proof of inequality (\ref{IG1}) that
\begin{equation}\label{IG2}
M_1 \sigma(\psi, v) \leq \rho(\psi, v)\leq M_2 \sigma(\psi, v),
\end{equation}
where $\sigma$ is given by
\[
\sigma(\psi, v)=\psi^2+v^2+(\psi')^2+(v')^2 +(\psi'')^2+(v'')^2.
\]
If we had a minimizing sequence $(\psi_n, v_n)_n\subset X$ for $S(\omega)$, meaning that $P(\psi_n, v_n)=0$ and $J_{\omega}(\psi_n, v_n) \to S(\omega)$, we consider the sequence of measures $\{\nu_n\}_n$ given by
\[
\nu_n(A)= \int_A \rho(\psi_n, v_n) \,dx,
\]
which is such that $\nu_n(\R) \to L$, according to the Lemma \ref{lem-r}. Now, from Lion's concentration-compactness principle, there exists a subsequence of $\{\nu_{n}\}_n$ (which will be denoted with the same index) such that one of the following three conditions holds:

\vspace{0.2cm}

\nd \textbf{(1) Compactness}: There exists a sequence $\{x_{n}\}_n\subset \mathbb{R}$ such that for any $\gamma>0$ there exists a radius $R>0$ such that
\begin{equation*}
\int_{B_R(x_{n})}d\nu_{n}\geq L- \gamma,
\end{equation*}
 for all $n$.
 
\nd  \textbf{(2)Vanishing}: For all $R>0$, it holds:
\begin{equation*}
\lim\limits_{n\rightarrow\infty} \left(\sup\limits_{y\in \mathbb{R}}\int_{\left|x-y\right|\leq R}d\nu_{n}\right)=0.
\end{equation*}
\nd \textbf{(3) Dichotomy}: There exists $\theta\in \left(0, L\right)$ such that for any $\gamma>0$, there exist a positive number $R$ and a sequence $\{x_{n}\}_n\subset \R $ with the following property: Given $R^{\prime}>R$, there are nonnegative measures $\nu_{n,1}$ and $\nu_{n,2}$ such that:

a) \ $0\leq \nu_{n,1}+\nu_{n,2}\leq \nu_{n}$;\\

b) \  $supp(\nu_{n,1})\subset B_R(x_{n})$ and $supp(\nu_{n,2})\subset \R\setminus B_{R^{\prime}}(x_{n})$;\\

c) \ $ \limsup\limits_{n \rightarrow \infty}\left( \left|\theta -\int_{\R}d\nu_{n,1}\right|+\left|(L -\theta) -\int_{\R}d\nu_{n,2}\right|\right)\leq \gamma$.

The first step is to rule out vanishing.
\begin{lem}\label{n-v}
Vanishing is impossible.
\end{lem}
\begin{proof} We argue by contradiction. Assuming that vanishing is true, so for $R=1$, we have that,
\begin{equation*}
\lim\limits_{n\rightarrow \infty}\sup\limits_{y \in \mathbb{R}}\int_{B_1(y)}\rho(\psi_{n}, v_n)\,dx=0.
\end{equation*}
Thus, given $\ep>0$, there is $n_0 \in \N$ such that for $n\geq n_0$ we have that
\begin{equation*}%\label{van-ep}
\sup\limits_{y \in \mathbb{R}}\int_{B_1(y)}\rho(\psi_{n}, v_n)\,dx<\ep.
\end{equation*}
Now, we recall that  $H^{1}(\mathbb{R})\hookrightarrow C_{b}(\mathbb{R})$, so, we have that
\begin{equation*}
\int_{B_1(y)}\left|\psi_{n}\right|^{p+2}\,dx\leq\left(\int_{B_1(y)}\left|\psi_{n}\right|^{2}\,dx\right)^{\frac{1}{2}}
\left(\int_{B_1(y)}\left|\psi_{n}\right|^{2\left(p+1\right)}dx\right)^{\frac{1}{2}}.
\end{equation*}
We can cover $\mathbb{R}$ with intervals of the form $[k, k + 1]$, where $k$ is an integer, ensuring that any point in $\mathbb{R}$ is contained in at most two intervals. By summing over these intervals and applying the inequality above, we obtain the following inequality:
\begin{equation}\label{DIFX}
\begin{split}
\int_{\R}(\left|\psi_{n}\right|^{p+2}+\left|v_{n}\right|^{p+2})\,dx & \leq  2\sup\limits_{y\in \R}\left(\int_{B_1(y)}\rho(\psi_{n}, v_n)\,dx\right)^{\frac{1}{2}}\left(\left\|\psi_{n}\right\|_{L^{2(p+1)}\left(\R\right)} +\left\|v_{n}\right\|_{L^{2(p+1)}\left(\R\right)}\right) \\
&\leq 2C\sup\limits_{y\in \R}\left(\int_{B_1(y)}\rho(\psi_{n}, v_n)\,dx\right)^{\frac{1}{2}}\left(\left\|\psi_{n}\right\|^{(p+1)}_{H^{1}(\R)}
+\left\|v_{n}\right\|^{(p+1)}_{H^{1}(\R)}\right) .
\end{split}
\end{equation}
A similar argument gives us that
\begin{equation}\label{DIFXX}
\begin{split}
\int_{\R}(\left|\partial_{x}\psi_{n}\right|^{p+2}+&\left|\partial_{x}v_{n}\right|^{p+2})\,dx\\
&\leq2C\sup\limits_{y\in R}\left(\int_{B_1(y)}\rho(\psi_{n}, v_n)\,dx\right)^{\frac{1}{2}}\left(\left\|\partial_{x}\psi_{n}\right\|^{\left(p+1\right)}_{H^{1}(\mathbb{R})}
+\left\|\partial_{x}v_{n}\right\|^{\left(p+1\right)}_{H^{1}(\mathbb{R})}\right).
\end{split}
\end{equation}
In other words, due to \eqref{DIFX} and \eqref{DIFXX}, we have $\psi_n, v_n, (\psi_n)_{x}, (v_n)_x, (\psi_n)_{xx}$ and $(v_n)_{xx}$ going to $0$ in $L^q(\R)$, for $q>2$. From this fact, we conclude that
\[
\lim_{n\to \infty}K(\psi_n, v_n)= \lim_{n\to \infty}N(\psi_n, v_n)=0.
\]
In fact, from Lemma \ref{esf-lev-2021} (vii), we have
\begin{multline*}
|K(\psi_n, v_n)|+|N(\psi_n, v_n)|\\ \leq2C\sup\limits_{y\in R}\left(\int_{B_1(y)}\rho(\psi_{n}, v_n)\,dx\right)^{\frac{1}{2}}\left(\left\|\psi_{n}\right\|^{q_1+1}_{H^{2}(\mathbb{R})}
+\left\|v_{n}\right\|^{q_1+1}_{H^{2}(\mathbb{R})}+\left\|\psi_{n}\right\|^{q_2+1}_{H^{2}(\mathbb{R})}
+\left\|v_{n}\right\|^{q_2+1}_{H^{2}(\mathbb{R})}\right),
\end{multline*}
where we are replacing $p$ with $q_j$ in previous inequality. In other words, we have shown that $\lim_{n\to \infty}I_{\omega}(\psi_n, v_n)=0$ due to the fact that $P(\psi_n, v_n)=0$. Moreover, we also have 
\[
S(\omega)=\lim_{n\to \infty} J_{\omega}(\psi_n, v_n)=0,
\]
which is a contradiction. So, we see that vanishing is ruled out, and the Lemma \ref{n-v} is shown. 
\end{proof}

Now, we are interested in ruling out dichotomy. In this case, we can choose a sequence $\gamma_{n}\rightarrow 0$ and the corresponding sequence $R_{n} \rightarrow \infty$, such that, passing to a subsequence with the same index, we can assume:
\begin{equation}\label{D}
supp(\nu_{n,1})\subset B_{R_n}(x_n), \quad supp(\nu_{n,2})\subset \R \setminus B_{2R_n}(x_n)
\end{equation}
and
\begin{equation}\label{E}
\limsup\limits_{n\rightarrow \infty}\left(  \left|\theta-\int_{\R}d\nu_{n,1}\right|+\left|(L-\theta)- \int_{\R}d\nu_{n,2}   \right|  \right)=0.
\end{equation}
Next, we establish a splitting result, 
\begin{lem}\label{split} (Splitting of a minimizing sequence)
Fix $\phi \in C^{\infty}_{0}(\mathbb{R}, \mathbb{R}^{+})$ such that $supp(\phi)\subset(-2, 2)$ and $\phi \equiv 1 \text{ on } (-1, 1)$ and define
\begin{equation*}
\psi_{n,1}=\psi_{n}\phi_{n}, \quad  \psi_{n,2}=\psi_{n}(1-\phi_{n}), \quad
 v_{n,1}=v_{n}\phi_{n},  \quad v_{n,2}=v_{n}(1-\phi_{n}),
\end{equation*}
where $\phi_n$ is given by
\begin{equation*}
\phi_{n}(x)= \phi\left(\frac{x-x_{n}}{R_{n}}\right).
\end{equation*}
Then, as  $n\to \infty$, we have the following splitting
\begin{align*}
I_{\omega}(\psi_{n}, v_n) & =I_{\omega}(\psi_{n,1}, v_{n,1})+ I_{\omega}( \psi_{n,2}, v_{n,2}) +o(1),\\
K(\psi_{n}, v_{n})&=K(\psi_{n,1}, v_{n,1})+  K(\psi_{n,2}, v_{n,2})+o(1), \\
J_{\omega}(\psi_{n}, v_n) & =J_{\omega}(\psi_{n,1}, v_{n,1})+ J_{\omega}( \psi_{n,2}, v_{n,2}) +o(1),\\
P_{\omega}(\psi_{n}, v_{n})&=P_{\omega}(\psi_{n,1}, v_{n,1})+  P_{\omega}(\psi_{n,2}, v_{n,2})+o(1),\\
N(\psi_{n}, v_{n})&=N(\psi_{n,1}, v_{n,1})+ 
N(\psi_{n,2}, v_{n,2})+o(1).
\end{align*}
\end{lem}
\begin{proof} 
Let $A(n)$  be the annulus $A(n)= B_{2R_n}(x_n)\setminus B_{R_n}(x_n)$. Then, we have,
\begin{align*}
\int_{A(n)}\rho(\psi_{n}, v_n)\,dx
=&\left\{  \int_{\R}- \int_{B_{R_n}(x_n)}- \int_{\R \setminus B_{2R_n}(x_n)}\right\}\rho(\psi_{n}, v_n)\,dx\\
 =& \int_{\R} \rho(\psi_{n}, v_n)\,dx-L +\theta-  \int_{B_{R_n}(x_n)}\rho(\psi_{n}, v_n)\,dx+\left(L-\theta\right)\\
 & - \int_{\R \setminus B_{2R_n}(x_n)} \rho(\psi_{n}, v_n)\,dx\\
\leq& \int_{\R} \rho(\psi_{n}, v_n)\,dx-L +\theta-  \int_{supp(\nu_{n, 1})}d\nu_{n,1}+\left(L-\theta\right)
- \int_{supp(\nu_{n, 2})} d\nu_{n,2}\\
 \leq&  \left(  \int_{\R} \rho(\psi_{n}, v_n)\,dx-L\right)+\left|\theta-\int_{\R}d\nu_{n,1}\right|+\left|(L-\theta)- \int_{\R}d\nu_{n,2}\right|.
\end{align*}
So, from Lemma \ref{lem-r} we have 
$$
\lim_{n\to \infty}\int_{\R} \rho(\psi_{n}, v_n)\,dx\rightarrow L.
$$ 
This convergence, together with conditions \eqref{D} and \eqref{E}, gives that 
\begin{equation} \label{r-n}
\limsup\limits_{n\rightarrow \infty}\left(\int_{A(n)}\rho(\psi_{n}, v_n)\,dx\right)=0.
\end{equation}
Using the previous fact, we conclude that
\[
\left| I_{\omega}(\psi_{n}, v_n) -I_{\omega}(\psi_{n,1}, v_{n,1})- I_{\omega}( \psi_{n,2}, v_{n,2}) \right|= \int_{A(n)}\rho(\psi_{n}, v_n) \,dx=o(1). 
\]
On the other hand, we also have that
\begin{multline*}
\left| K(\psi_{n}, v_n) -K(\psi_{n,1}, v_{n,1})- K( \psi_{n,2}, v_{n,2}) \right|\\
\leq C\int_{A(n)}(|(\psi_{n}, \partial_x\psi_{n}, v_n, \partial_x v_n)|^{q_1+2}+ |(\psi_{n}, \partial_x\psi_{n}, v_n, \partial_x v_n)|^{q_2+2})=o(1),
\end{multline*}
where we are using the estimate (\ref{r-n}) and Lemma \ref{esf-lev-2021} (iii). 
Other estimates follow from a similar argument. 
\end{proof}

With this in hand, we will prove that the dichotomy does not hold.
\begin{lem}\label{n-d}
Dichotomy is not possible.
\end{lem}
\begin{proof}
First, we need to note that
\[
\lim_{n\to \infty} I_{\omega}(\psi_{n,1}, v_{n,1})=L_1\geq 0 \quad \text{ and }\quad \ \ \lim_{n\to \infty} I_{\omega}(\psi_{n,2}, v_{n,2})=L-L_1=L_2\geq 0.
\]
From condition (c) in Dichotomy, we have that $L_1, L_2>0$. In fact, 
\[
0< \theta= \lim_{n\to \infty} \int_{\R}d \nu_{n,1} \leq \lim_{n\to \infty} \int_{B_R(x_n)} \rho(\psi_{n,1}, v_{n,1})dx =L_1.  
\]
In particular, we have that
\[
\lim_{n\to \infty} N(\psi_{n,1}, v_{n,1})=L_1>0, \quad \text{and} \quad \lim_{n\to \infty} N(\psi_{n,2}, v_{n,2})=L_2>0.
\]
From the fact that the sequence $(\psi_{n,1}, v_{n,1})_n$ is bounded in $H^2(\R)$, we have that the sequence $\lambda_j=P_{\omega}(\psi_{n,j}, v_{n,j})_n$ for $j=1, 2$ is bounded. So,   using  a subsequence if necessary, we are allowed to define, $\lambda_{n,1}=P_{\omega}(\psi_{n,1}, v_{n,1})$ and  $\lambda_{n,2}=P_{\omega}(\psi_{n,2}, v_{n,2})$, and the limits
$$
\lambda_1=\lim\limits_{n \rightarrow \infty}\lambda_{n,1} \quad \text{and} \quad \lambda_2= \lim\limits_{n \rightarrow \infty}\lambda_{n,2}.
$$
We first consider the case $\lambda_1=0$ and $\lambda_2=0$. 
Suppose that for some subsequence of $\{(\psi_{n,1}, v_{n,1})\}$, still denoted by the same index, we had that $P_{\omega}(\psi_{n,1}, v_{n,1})>0$, then from Lemma \ref{esf-lev-2021}, we have that
\begin{align*}
P_{\omega}(\alpha(\psi_{n,1}, v_{n,1}))&\leq \alpha^2 N(\psi_{n,1}, v_{n,1})\left(\frac{I_{\omega}(\psi_{n,1}, v_{n,1})}{N(\psi_{n,1}, v_{n,1})}-\alpha^p \right)\\
&\leq  \alpha^2 N(\psi_{n,1}, v_{n,1})\left(1+o(1)-\alpha^p\right)<0,
\end{align*}
for $\alpha>1$. Thus, there is $(\alpha_n)_n$ with $\alpha_n \to 1^{+}$ such that $P_{\omega}(\alpha_n(\psi_{n,1}, v_{n,1}))=0$. Now, from the splitting result, we have that
\begin{align*}
J_{\omega}(\psi_{n}, v_{n})=&\frac12 I_{\omega}(\psi_{n,1}, v_{n,1}) - K(\psi_{n,1}, v_{n,1})+ J_{\omega}(\psi_{n,2}, v_{n,2})+o(1),\\
=& \frac{1}{\alpha_n^2}\left(\frac12 I_{\omega}(\alpha_n(\psi_{n,1}, v_{n,1})) - K(\alpha_n(\psi_{n,1}, v_{n,1}))\right)\\
&  +\left(\frac{1}{\alpha_n^2} K(\alpha_n(\psi_{n,1}, v_{n,1})) - K(\psi_{n,1}, v_{n,1})\right)+ J_{\omega}(\psi_{n,2}, v_{n,2})+o(1),\\
\geq& \frac{1}{\alpha_n^2}S(\omega)+\frac{1}{\alpha_n^2} \left(K(\alpha_n(\psi_{n,1}, v_{n,1})) - K(\psi_{n,1}, v_{n,1})\right)  +J_{\omega}(\psi_{n,2}, v_{n,2})+o(1).
\end{align*}
Using that $J_{\omega}(\psi_{n}, v_{n})=J_{\omega, p}(\psi_{n}, v_{n})$ and taking limit as $n\to \infty$, we conclude that
\[
0\geq \lim_{n\to \infty}J_{\omega}(\psi_{n,2}, v_{n,2})\geq \lim_{n\to \infty}\left(\frac12 I_{\omega}(\psi_{n,2}, v_{n,2})- \frac{1}{p+2}N(\psi_{n,2}, v_{n,2})\right)=\left(\frac{p}{2(p+2)}\right)L_2,
\]
which is a contradiction since $L_2>0$. So, we already have that  $P_{\omega}(\psi_{n,i}, v_{n,i})\leq 0$ for $i=1, 2$. On the other hand, Lemma \ref{inf} item ii) gives that  $J_{\omega,p}(\psi_{n,i}, v_{n,i})\geq S(\omega)$, for $i=1,2$, which ensures us a contradiction $S(\omega)\leq 0$ due to the splitting result.

Finally, we assume that $\lambda_1>0$ and $\lambda_2<0$. So, for $n\in \N$ large enough, we have $P_{\omega}(\psi_{n,1}, v_{n,1})>0$ and  $P_{\omega}(\psi_{n,2}, v_{n,2})<0$. Then, we see that
\begin{align*}
J_{\omega}(\psi_{n}, v_{n})&= J_{\omega,p}(\psi_{n,1}, v_{n,1}) + J_{\omega,p}(\psi_{n,2}, v_{n,2})+o(1),\\
&\geq J_{\omega,p}(\psi_{n,1}, v_{n,1}) + S(\omega)+o(1),
\end{align*}
which implies that
\begin{align*}
0\geq \lim_{n\to \infty}J_{\omega, p}(\psi_{n,1}, v_{n,1})&= \lim_{n\to \infty}\left( \frac12 I_{\omega}(\psi_{n,1}, v_{n,1})-K(\psi_{n,1}, v_{n,1})- \frac{1}{p+2}P(\psi_{n,1}, v_{n,1})\right)\\
&= \lim_{n\to \infty}\left( \frac{p}{2(p+2)}I_{\omega}(\psi_{n,1}, v_{n,1})-K(\psi_{n,1}, v_{n,1})+\frac{1}{p+2}N(\psi_{n,1}, v_{n,1})\right)\\
&\geq \lim_{n\to \infty}\frac{p}{2(p+2)}I_{\omega}(\psi_{n,1}, v_{n,1})=\left(\frac{p}{2(p+2)}\right)L_1,
\end{align*}
which is a contradiction with $L_1>0$. In other words, we have ruled out Dichotomy, and Lemma \ref{n-d} is proven.
\end{proof}
Now, we use the compactness property to determine the existence of non-trivial traveling waves for the generalized Boussinesq system in the cases $b=d$ and $b_2=d_2$, that is, we are in a position to prove the main result of the work.

\begin{proof}[\bf Proof of Theorem \ref{exist}] From Lion's concentration-compactness principle, after having ruled out dichotomy and vanishing, we have compactness property which gives the existence of a sequence $\left\{y_{n}\right\}_n\subset \mathbb{R}$ such that, for all  $\epsilon >0$ there exists  $R(\epsilon)>0$ satisfying
\begin{equation*}
\int_{\left|x-y_{n}\right|<R(\epsilon)}\rho(\psi_{n}, v_n)\,dx \geq \int_{\mathbb{R}}\rho(\psi_{n}, v_n)\,dx - \epsilon.
\end{equation*}
The previous inequality is equivalent to the following one
\begin{equation*}
\int_{\left|x\right|<R(\epsilon)}\rho(\tilde \psi_{n}, \tilde v_n)\,dx \geq \int_{\mathbb{R}}\rho(\tilde \psi_{n}, \tilde v_n)\,dx - \epsilon,
\end{equation*}
when translated to origin, where  $\tilde \eta$ means $\tilde \eta(x)=\eta(x+y_{n})$. In other words, we have that
\begin{equation*}
\int_{\left|x\right|\geq R(\epsilon)}\rho(\tilde \psi_{n}, \tilde v_n)\,dx \leq \epsilon.
\end{equation*}
From the definition of $\rho$ in (\ref{RHO}) and inequality (\ref{IG2}), we get that
\begin{equation*}
M_1\int_{\left|x\right|\geq R(\epsilon)}\left(\left(\partial^{j}_{x}\tilde \psi_n\right)^{2}+\left(\partial^{j}_{x}\tilde v_n\right)^{2}\right)\,dx\leq\int_{\left|x\right|\geq R(\epsilon)}\rho(\tilde \psi_{n}, \tilde v_n)\,dx \leq \epsilon.
\end{equation*}
From this fact, we conclude that
\begin{align}\label{RRRR}
\int_{\mathbb{R}}\left(\partial^{i}_{x}\tilde \psi_n\right)^{2}\,dx & \leq \int_{\left|x\right|\leq R(\epsilon)}\left(\partial^{i}_{x} \tilde \psi_{n}\right)^{2}\,dx+ \frac{\epsilon}{M_1},
\end{align}
and
\begin{align}\label{RRRR1}
\int_{\mathbb{R}}\left(\partial^{i}_{x}\tilde v_n\right)^{2}\,dx & \leq \int_{\left|x\right|\leq R(\epsilon)}\left(\partial^{i}_{x} \tilde v_{n}\right)^{2}\,dx+ \frac{\epsilon}{M_1}.
\end{align}
Now, we will establish that
\[
\partial^{i}_{x}\tilde \psi_{n}\rightarrow \partial^{i}_{x}\psi_{0}, \ \ \partial^{i}_{x}\tilde v_{n}\rightarrow \partial^{i}_{x}v_{0}\ \ \mbox{in $L^{2}\left(\mathbb{R}\right)$, \quad \text {for $i=0, 1, 2$}. }
\]
Observe that for any bounded  open interval $I$, and any $ r > 1$, we have that the embedding $W^{1, r}(I) \hookrightarrow C^{\alpha}(\bar I)$ is compact for $\alpha\in \left[0, 1-\frac{n}{r}\right) $, and so the embedding $W^{k, r}(I) \hookrightarrow L^q(I)$, for $q\geq 1 $, is valid, since $C^{\alpha}(I)\hookrightarrow  L^q(I)$  for $q\geq 1$ and for $k\geq 1$ the operators $W^{k, r}(I) \hookrightarrow W^{1, r}(I)$  are bounded. In particular, we have for $0\leq i\leq 2$ that
\begin{align}\label{r-r}
\partial_x^i \tilde \psi_n \rightharpoonup \partial_x^i\psi_0 &, \ \ \mbox{in} \ \ H^{2-i}(\R),
\end{align}
and
\begin{align}\label{r-r1}
\partial_x^i \tilde v_n \rightharpoonup \partial_x^i v_0 &, \ \ \mbox{in} \ \ H_{loc}^{1-i}(\R).
\end{align}
From the convergence \eqref{r-r}, using Fatou's Theorem and thanks the inequality \eqref{RRRR}, we ensure that
\begin{align*}
\left\|\partial^{i}_{x}\psi_{0}\right\|^{2}_{L^{2}(\mathbb{R})}  &\leq \liminf\limits_{n\to \infty}\left\|\partial^{i}_{x}\tilde \psi_{n}\right\|^{2}_{L^{2}(\mathbb{R})}\\
&\leq \liminf\limits_{n\to \infty}\left(  \int_{\left|x\right|\leq R(\epsilon)} \left(\partial^{i}_{x}\tilde  \psi_{n}\right)^{2}\,dx\right)+ \frac{\epsilon}{M_1}\\
&\leq \int_{\left|x\right|\leq R(\epsilon)} \left(\partial^{i}_{x}\psi_{0}\right)^{2}\,dx+ \frac{\epsilon}{M_1}\\
&\leq \left\|\partial^{i}_{x}\psi_{0}\right\|^{2}_{L^{2}(\mathbb{R})}+ \frac{\epsilon}{M_1},
\end{align*}
 for $0\leq i\leq 2$. Thus, 
$$  \liminf\limits_{n\to \infty}\left\|\partial^{i}_{x}\tilde \psi_{n}\right\|^{2}_{L^{2}(\mathbb{R})}=\left\|\partial^{i}_{x}\psi_{0}\right\|^{2}_{L^{2}(\mathbb{R})}.$$
  Analogously, using the convergence \eqref{r-r1}, once again, Fatou's Theorem and now the inequality \eqref{RRRR1}, we get
 \begin{align*}
\left\|\partial^{i}_{x}v_{0}\right\|^{2}_{L^{2}(\mathbb{R})} \leq  \liminf\limits_{n\to \infty}\left\|\partial^{i}_{x}\tilde v_{n}\right\|^{2}_{L^{2}(\mathbb{R})}\leq\left\|\partial^{i}_{x}v_{0}\right\|^{2}_{L^{2}(\mathbb{R})}+ \frac{\epsilon}{M_1},
\end{align*}
that is, 
$$  \liminf\limits_{n\to \infty}\left\|\partial^{i}_{x}\tilde v_{n}\right\|^{2}_{L^{2}(\mathbb{R})}=\left\|\partial^{i}_{x}v_{0}\right\|^{2}_{L^{2}(\mathbb{R})}.$$
Thanks to both inequalities, we have shown  that
$$
\left\|\partial^{i}_{x}\tilde \psi_{n}\right\|_{L^{2}(\mathbb{R})} \rightarrow \left\|\partial^{i}_{x}\psi_{0}\right\|_{L^{2}(\mathbb{R})},$$
and
$$ \left\|\partial^{i}_{x}\tilde v_{n}\right\|_{L^{2}(\mathbb{R})} \rightarrow \left\|\partial^{i}_{x} v_{0}\right\|_{L^{2}(\mathbb{R})},
$$
for $0\leq i\leq 2$. 

On the other hand, we also have that $\partial^{i}_{x}\tilde \psi_{n}\rightharpoonup \partial^{i}_{x}\tilde \psi_{0}$ and $\partial^{i}_{x}\tilde v_{n}\rightharpoonup \partial^{i}_{x}v_{0}$ weakly in $L^{2}(\mathbb{R})$ for $0\leq i\leq 2$. Thus, using these facts, we conclude that
$$
\partial^{i}_{x} \tilde \psi_{n} \rightarrow \partial^{i}_{x}\psi_{0}\quad \partial^{i}_{x} \tilde v_{n} \rightarrow \partial^{i}_{x}v_{0}\quad \text{in} \quad  L^{2}(\mathbb{R}),
$$
for $0\leq i\leq 2$. Moreover, we also have that $0=P_{\omega}(\psi_{n}, v_n)=P_{\omega}(\tilde \psi_{n}, \tilde v_n)\rightarrow P_{\omega}(\psi_{0}, v_0)=0$, since we have strong convergence and  that $P_{\omega}$ is Lipschitz (continuous). We also  have that  $S(\omega) \leq J_{\omega}(\psi_0, v_0)$. Thus, the weak convergence of $(\tilde \psi_n, \tilde v_n)$ to $(\psi_0, v_0)$ in $X$ and the weak lower semi-continuity of functional $J_{\omega}$, yields that
\[
 J_{\omega}(\psi_0, v_0) \leq \liminf J_{\omega}(\tilde \psi_n, \tilde v_n) =\liminf J_{\omega}(\psi_n, v_n) = S(\omega),
\]
meaning that $J_{\omega}(\psi_0, v_0)=S(\omega)$. In other words, $(\psi_0, v_0)$ is in fact a minimizer for $S(\omega)$ and $(\tilde \psi_{n}, \tilde v_n)\rightarrow (\psi_{0}, v_0)$ in $X$, as claimed.
\end{proof}
Now, we define the set of ground states as 
\begin{equation}\label{ground}
\mathcal G_{\omega}=\left\{(\psi, v)\in X\setminus\{0\}: J_{\omega}(\psi, v)=S(\omega),  \  \mbox{and} \ P_{\omega}(\psi, v)=0\right\}.
\end{equation}

\begin{lem}\label{estd}
Let $0<|\omega|<\min\left\{1,\frac{{-a}}{b}, \frac{{-c}}{b},\frac{a_2}{b_2},\frac{c_2}{b_2}\right\}$ and $\mathcal G_{\omega}$ defined by \eqref{ground}.

\nd i) Then $S({\omega})$ is uniformly bounded.

\nd ii) If $0<{\omega}_1 <\omega_2<1$ and $(\psi,v)\in\mathcal{G}_{\omega}$, we find that $I_{2,{\omega}}(\psi,v)$ is uniformly bounded on $[{\omega}_1,{\omega}_2]$.

\nd iii) For $(u, v)\in \mathcal G_{\omega}$, we have $S'(\omega)=-I_{2}(u, v)$. Here, $I_2$ is defined by \eqref{I2a}.
\end{lem}
\begin{proof} Consider $I_1$ and $I_2$ defined by \eqref{I1a} and \eqref{I2a}, respectively.

i) Note that $I_{\omega}(\psi, v)=I_1(\psi, 0)+I_1(0,v)-2\omega I_2(\psi, v)$. So, we choose $\psi\in H^2(\R)$ such that $I_1(\psi, 0)-N(\psi, 0)=0$. Thus, $P_{\omega}(\psi,0)=I_1(\psi, 0)- N(\psi,0)=0$. From this and Lemma \ref{esf-lev-2021} (vii), we get that
\[
S({\omega})\leq J_{\omega}(\psi, 0)\leq  I_1(\psi, 0)-K(\psi, 0)\leq  ||\psi||_{H^2}^2+C\left(||\psi||_{H^2}^{q_1+2}+||\psi||_{H^2}^{q_2+2}\right),
\]
which implies  that $S({\omega})\leq C$, where $C>0$ is independent of $0<\omega<1$.

\nd ii) For $0<{\omega}_1<\omega <{\omega}_2<1$ and $(\psi,v)\in\mathcal{G}_{\omega}$, we have that
\[
C\geq S({\omega})=J_{\omega}(\psi,v)\geq \frac12 I_{\omega}(\psi,v)- \frac1{p+2} N(\psi,v)\geq \left(\frac{p}{2(p+2)}\right) I_{\omega}(\psi,v),
\]
which implies that
\[
I_{2, \omega}(\psi,v)\leq C|\omega|||(\psi, v)||_X^2\leq M_2(a, b, c, b_2, c_2\omega_1)|\omega_2| I_{\omega}(\psi,v),
\]
where $M_2(\omega_1, a, a_2, b, b_2, c, c_2)=\frac{1}{M_1(\omega_1, a, a_2, b, b_2, c, c_2)}$ is defined in Lemma \ref{inf}.

\nd iii) Let ${\omega}_1<{\omega}_2$ and $(u^{{\omega}_i},v^{{\omega}_i})\in\mathcal{G}_{{\omega}_i}$. In the case $I_{\omega_1}(u^{\omega_2},v^{\omega_2})\leq I_{\omega_2}(u^{\omega_2},v^{\omega_2})$, and we have that  $P_{\omega_1}(u^{\omega_2},v^{\omega_2})\leq P_{\omega_2}(u^{\omega_2},v^{\omega_2})=0$. This implies that
\begin{equation}\label{Sprima-1}
\begin{split}
S(\omega_1)&\leq J_{\omega_1,p}(u^{\omega_2},v^{\omega_2}), \\
& \leq J_{\omega_2,p}(u^{\omega_2},v^{\omega_2})+ (\omega_2-\omega_1) I_{2}(u^{\omega_2},v^{\omega_2}), \\
& \leq S(\omega_2)+ (\omega_2-\omega_1) I_{2}(u^{\omega_2},v^{\omega_2}).
\end{split}
\end{equation}
Now, we consider the case $I_{\omega_1}(u^{\omega_2},v^{\omega_2})> I_{\omega_2}(u^{\omega_2},v^{\omega_2})$. So, for  $\alpha_2=\left(\frac{I_{\omega_1}(u^{\omega_2},v^{\omega_2})}{I_{\omega_2}(u^{\omega_2},v^{\omega_2})}
\right)^{\frac1p}>1$, we have
\begin{align*}
P_{{\omega}_1}(\alpha_2(u^{\omega_2},v^{\omega_2}))&= \alpha^2_2 I_{\omega_1}(u^{\omega_2},v^{\omega_2})- N(\alpha(u^{\omega_2},v^{\omega_2})),\\
&\leq \alpha^2_2 I_{\omega_1}(u^{\omega_2},v^{\omega_2})- \alpha^{p+2}_2N(u^{\omega_2},v^{\omega_2}),\\
&\leq \alpha^2_2\left(I_{\omega_1}(u^{\omega_2},v^{\omega_2})-  \alpha^{p}_2I_{\omega_2}(u^{\omega_2},v^{\omega_2})\right)\\
&\leq 0.
\end{align*}
From this fact, we conclude that
\begin{align*}
S(\omega_1)&\leq J_{{\omega}_1, p}(\alpha_2(u^{\omega_2},v^{\omega_2}))\\
&\leq \alpha_2^2 S(\omega_2)+\alpha_2^2(\omega_2-\omega_1)I_{2}(u^{\omega_2},v^{\omega_2})+\alpha_2^2 K(u^{\omega_2},v^{\omega_2})-K(\alpha_2(u^{\omega_2},v^{\omega_2}))\\& -\frac1{p+2}P_{\omega}(\alpha_2(u^{\omega_2},v^{\omega_2})).
\end{align*}
Now, pick $g(\alpha)=\alpha^2 K(u^{\omega_2},v^{\omega_2})-K(\alpha(u^{\omega_2},v^{\omega_2}))$. Then, we see that $g(1)=0$ and $g'(1)=-2S(\omega_2)$. So, for $\omega_1-\omega_2 \sim 0$, we ensure that 
\[
g(\alpha_2)=\alpha_2^2 K(u^{\omega_2},v^{\omega_2})-K(\alpha_2(u^{\omega_2},v^{\omega_2}))= -2S(\omega_2)(\alpha_2-1)+ O((\alpha_2-1)^2),
\]
where we are using that
\[
\alpha_2-1=\frac{2 (\omega_2-\omega_1)}{p} \frac{I_{2}(u^{\omega_2},v^{\omega_2})}{I_{\omega_2}(u^{\omega_2},v^{\omega_2})}+O((\omega_2-\omega_1)^2).
\]
On the other hand, we see for 
$\omega_1-\omega_2 \sim 0$ that $P_{\omega}(\alpha_2(u^{\omega_2},v^{\omega_2}))=O(|\alpha_2-1|)$ since
\[
\lim_{\alpha_2\to 1}P_{\omega}(\alpha_2(u^{\omega_2},v^{\omega_2}))=P_{\omega}((u^{\omega_2},v^{\omega_2}))=0. 
\]
From these facts, we have
\begin{equation}\label{Sprima-2}
\begin{split}
S(\omega_1) \leq& S(\omega_2) +(\alpha_2^2-1) S(\omega_2)+\alpha_2^2(\omega_2-\omega_1)I_{2}(u^{\omega_2},v^{\omega_2})\\&-2S(\omega_2)(\alpha_2-1)+ O(|\alpha_2-1|),\\
\leq &S(\omega_2) +(\alpha_2^2-1) S(\omega_2)\left(\frac{\alpha_2-1}{\alpha_2+1} \right)\\&+\alpha_2^2(\omega_2-\omega_1)I_{2}(u^{\omega_2},v^{\omega_2})+ O(|\alpha_2-1|).
\end{split}
\end{equation}
Finally, from (\ref{Sprima-1}) and (\ref{Sprima-2}), we conclude that
\[
(S')^{-}(w_2)=\lim_{\omega_1\to \omega_2^{-}}\frac{S(\omega_1)-S(\omega_2) }{\omega_1-\omega_2}\geq -I_{2}(u^{\omega_2},v^{\omega_2}).
\]
In a similar way, interchanging the role $\omega_1$ and $\omega_2$ in an appropriated way,  we see that
\[
(S')^{+}(w_1)=\lim_{\omega_2\to \omega_1^{+}}\frac{S(\omega_2)-S(\omega_1) }{\omega_2-\omega_1}\leq -I_{2}(u^{\omega_1},v^{\omega_1}).
\]
Putting previous inequalities together, we conclude for any $(\psi, v)\in \mathcal G_{\omega}$ that
\[
S'(\omega)=-I_{2}(\psi, v),
\]
giving the result, and the Lemma \ref{estd} is achieved.
\end{proof}

\section{Numerical experiments} \label{sec3}

In this section, we introduce numerical solvers to approximate the solutions of the solitary wave equations \eqref{trav-eqs} and compute traveling wave solutions for a given wave velocity $\omega$, based on the parameter regime outlined in the previous section. Additionally, some of our experiments investigate scenarios where the wave velocity lies outside the previously established theoretical range.

\subsection{Homogeneous case}

When the nonlinear functions $H_1$ and $H_2$ in the Boussinesq system \eqref{1bbl} are homogeneous, we adapt the numerical solver from \cite{Esfahani-Levandosky-2021}, originally developed for the fifth-order scalar KdV equation, to approximate the solutions of the traveling-wave equations \eqref{trav-eqs}. This method uses a Fourier basis and incorporates stabilizing factors for wave elevation $\eta$ and fluid velocity $u$, ensuring the convergence of the iterative scheme employed in the numerical experiments presented here.

Taking Fourier transform of equations \eqref{trav-eqs}, we obtain
\begin{align*}
\begin{cases}
-\omega ( \hat{v}_k + d k^2 \hat{v}_k + d_2 k^4 \hat{v}_k ) + \hat{\psi}_k - c k^2 \hat{\psi}_k + c_2 k^4 \hat{\psi}_k &=\widehat{H_1}, \\
-\omega( \hat{\psi}_k + b k^2 \hat{\psi}_k + b_2 k^4 \hat{\psi}_k ) + \hat{v}_k - a k^2 \hat{v}_k + a_2 k^4 \hat{v}_k & = \widehat{H_2}.
\end{cases}
\end{align*}
We recall that $b=d>0$, $b_2 = d_2 >0$, $a, c <0$, $a_2, c_2 >0$, and $\widehat{H_1}, \widehat{H_2}$ are the Fourier transforms of the functions $H_1(\psi, \psi', \psi'', v, v', v'')$ and $H_2(\psi, \psi', \psi'', v, v', v'')$, respectively.

Rewriting these equations, we get
\begin{align*}
\begin{cases}
-\omega (1 + d k^2 + d_2 k^4) \hat{v}_k + (1-ck^2 + c_2 k^4) \hat{\psi}_k &= \widehat{H_1}, \\
(1-ak^2 + a_2 k^4)\hat{v}_k - \omega ( 1 + b k^2 + b_2 k^4) \hat{\psi}_k &= \widehat{H_2},
\end{cases}
\end{align*}
and thus, we can write
\begin{align}
&\hat{v}_k = \frac{ \widehat{H_1} D_{22} - \widehat{H_2}  D_{12} }{D_{11} D_{22} - D_{21} D_{12} }, \label{v_eq}
\end{align}
and
\begin{align}
&\hat{\psi}_k = \frac{ \widehat{H_2} D_{11} - \widehat{H_1} D_{21} }{ D_{11} D_{22} - D_{21} D_{12} }, \label{psi_eq}
\end{align}
where
\begin{equation*}
D = \begin{pmatrix}
D_{11} & D_{12} \\
D_{21} & D_{22} 
\end{pmatrix} =
\begin{pmatrix}
-\omega ( 1 + d k^2 + d_2 k^4 ) & 1 - c k^2 + c_2 k^4 \\
1 - a k^2 + a_2 k^4 & -\omega (1 + b k^2 + b_2 k^4 )
\end{pmatrix}.
\end{equation*}

On the other hand, by multiplying equation \eqref{v_eq} by $\hat{v}_k$ and equation \eqref{psi_eq} by $\hat{\psi}_k$, and then adding over $k$, we obtain
\begin{align*}
&P_1(v ) := \frac{ \sum_k  (D_{11} D_{22} - D_{21} D_{12} ) \hat{v}_k^2  }{ \sum_k (\widehat{H}_1 D_{22} - \widehat{H_2} D_{12}  ) \hat{v}_k } = 1,
\end{align*}
and
\begin{align*}
&P_2(\psi) := \frac{  \sum_k ( D_{11} D_{22} - D_{21} D_{12}  ) \hat{\psi}_k^2  }{  \sum_k ( \widehat{H_2} D_{11} - \widehat{H_1} D_{21} ) \hat{\psi}_k } = 1.
\end{align*}
To compute approximate solutions to the traveling wave equations, we select the initial values $\psi^0$ and $v^0$. Then, for $s \geq 0$, we define the sequence
\begin{align}
&\hat{v}_{k}^{s+1} = \frac{\widehat{H_1}(\tilde{\psi^s}, \tilde{\psi^s}', \tilde{\psi^s}'', \tilde{v^s}, \tilde{v^s}', \tilde{v^s}'') D_{22} - \widehat{H_2}(\tilde{\psi^s}, \tilde{\psi^s}', \tilde{\psi^s}'', \tilde{v^s}, \tilde{v^s}', \tilde{v^s}'')  D_{12}  }{ D_{11} D_{22} - D_{21} D_{12} }, \label{iteration_1} 
\end{align}
and
\begin{align}
&\hat{\psi}_k^{s+1} =  \frac{\widehat{H_2}(\tilde{\psi^s}, \tilde{\psi^s}', \tilde{\psi^s}'', \tilde{v^s}, \tilde{v^s}', \tilde{v^s}'') D_{11} - \widehat{H_1}(\tilde{\psi^s}, \tilde{\psi^s}', \tilde{\psi^s}'', \tilde{v^s}, \tilde{v^s}', \tilde{v^s}'')  D_{21}  }{ D_{11} D_{22} - D_{21} D_{12} }, \label{iteration_2}
\end{align}
where $\tilde{v}^s : = \alpha_s v^s$, $\tilde{\psi}^s := \beta_s \psi^s$, and $\alpha_s$, $\beta_s \in \mathbb{R}$ are solutions of the following equations
\begin{equation}\label{Pequations}
\begin{cases}
P_1(\alpha_s v^s ) = 1,\\
P_2(\beta_s \psi^s ) = 1.
\end{cases}
\end{equation}
We note that the quantities $\alpha_s$ and $\beta_s$ are stabilizing factors introduced to ensure the convergence of the iteration defined in \eqref{iteration_1} and \eqref{iteration_2}.

Suppose that the functions $H_1, H_2$ are both homogeneous of order $p+1$. Then, given that the parameters $\alpha_s$, $\beta_s$ are chosen at each iteration $s$ such that $P_1(\alpha_s v^s) = 1$ and 
$P_2(\beta_s \psi^s) = 1$, we obtain that
\begin{equation}\label{solit_curve}
\begin{cases}
&\frac{ \sum_k  (D_{11} D_{22} - D_{21} D_{12} ) \alpha_s^2 \hat{v}_k^2  }{ \sum_k \alpha_s^{p+1} (\widehat{H_1} D_{22} - \widehat{H_2} D_{12}  ) \alpha_s \hat{v}_k } = 1,\\
\\
&\frac{  \sum_k ( D_{11} D_{22} - D_{21} D_{12}  ) \beta_s^2 \hat{\psi}_k^2  }{  \sum_k \beta_s^{p+1} ( \widehat{H_2} D_{11} - \widehat{H_1} D_{21} ) \beta_s \hat{\psi}_k } = 1.
\end{cases}
\end{equation}
From equations \eqref{Pequations} and \eqref{solit_curve}, we obtain the explicit expressions for the parameters $\alpha_s, \beta_s$ given by 
\begin{equation*}
\begin{cases}
\alpha_s &= M_s^{1/p},\\
\beta_s &= N_s^{1/p},
\end{cases}
\end{equation*}
where
\begin{align*}
&M_s := \frac{ \sum_k  (D_{11} D_{22} - D_{21} D_{12} ) \hat{v}_k^2  }{ \sum_k  (\widehat{H}_1 D_{22} - \widehat{H_2} D_{12}  ) \hat{v}_k },
\end{align*}
and
\begin{align*}
&N_s := \frac{  \sum_k ( D_{11} D_{22} - D_{21} D_{12}  ) \hat{\psi}_k^2  }{  \sum_k ( \widehat{H_2} D_{11} - \widehat{H_1} D_{21} ) \hat{\psi}_k }.
\end{align*}

As a consequence, the iteration defined by equations \eqref{iteration_1} and \eqref{iteration_2}, can be rewritten as follows:
\begin{align}
&\hat{v}_{k}^{s+1} = M_s^{ \frac{p+1}{p} } \frac{\widehat{H_1}( {\psi^s}, {\psi^{s}}', {\psi^{s}}'', {v^s}, {v^{s}}', {v^{s}}'') D_{22} - \widehat{H_2}(\psi^s, {\psi^{s}}', {\psi^{s}}'', v^s, {v^{s}}', {v^{s}}'' )  D_{12}  }{ D_{11} D_{22} - D_{21} D_{12} }, \label{iteration_3}
\end{align}
and
\begin{align}
&\hat{\psi}_k^{s+1} = N_s^{\frac{p+1}{p} } \frac{\widehat{H_2}({\psi^s, \psi^s}', {\psi^s}'', {v^s}, {v^s}', {v^s}'') D_{11} - \widehat{H_1}( {\psi^s}, {\psi^s}',
{\psi^s}'', {v^s}, {v^s}', {v^s}'')  D_{21}  }{ D_{11} D_{22} - D_{21} D_{12} }. \label{iteration_4}
\end{align}
Next, we compute some approximate solitary wave solutions to the Boussinesq system \eqref{1bbl} for different model parameters using the numerical scheme introduced, with $H_1 = u^{p+1}$ and $H_2 = \eta^{p+1}$.

In Figure \ref{Solit_u_eta1}, we display the solitary wave solution $(u, \eta)$ computed using the iterative scheme \eqref{iteration_3}-\eqref{iteration_4}, with initial values given by
\begin{equation}
v^0(x) = \psi^0(x) = e^{-0.5 (x-a_0)^2}, \label{starting_values}
\end{equation}
with $a_0 = 100$.  We have used $N=2^{12}$ spectral points in the spatial domain, with the computational domain defined as the interval $[0,L]=[0,200]$.
We point out that the non-monotonicity for $x \geq 0$ of the solitary waves shown in Figure \ref{Solit_u_eta1} arises from an imbalance between the nonlinear and dispersive terms in the Boussinesq system \eqref{1bbl}. Similar behavior has already been observed in other nonlinear dispersive models, such as the modified Kawahara equation studied in \cite{Marinov}.

In Figure \ref{Solit_u_eta2}, we present the results obtained for different model parameters, using the same initial values as given in \eqref{starting_values}. Additionally,  Figure \ref{Solit_u_eta3} considers a case where the wave velocity is outside the theoretical interval of existence established in the previous section. We take the wave velocity as $\omega = 0.4$, which, for the parameters chosen for the adopted model, is outside the interval of existence $0< |\omega| < \min\left\{1,\frac{-a}{b}, \frac{-c}{b}, \frac{a_2}{b_2},\frac{c_2}{b_2}\right\} = 0.25$. The computational domain is $[0,L]= 150$, and $a_0 = 75$ in \eqref{starting_values}. All other parameters are the same as those in the previous numerical experiments.

\begin{figure}[h!]
     \centering
     \begin{subfigure}[b]{0.45\textwidth}
         \centering
         \includegraphics[width=\textwidth, height=0.8\textwidth]{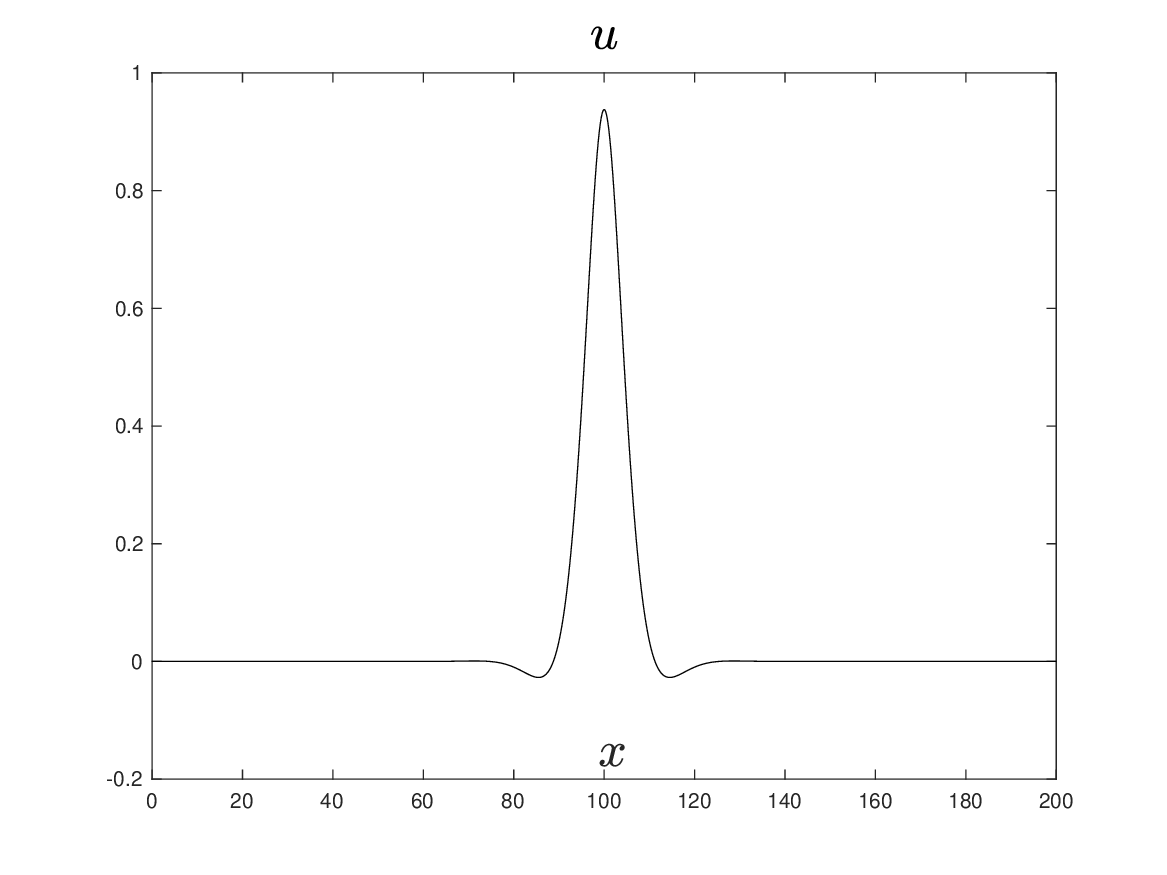}
         %\caption{}
         %\label{fig:y equals x}
     \end{subfigure}
     \begin{subfigure}[b]{0.45\textwidth}
         \centering
         \includegraphics[width=\textwidth, height=0.8\textwidth]{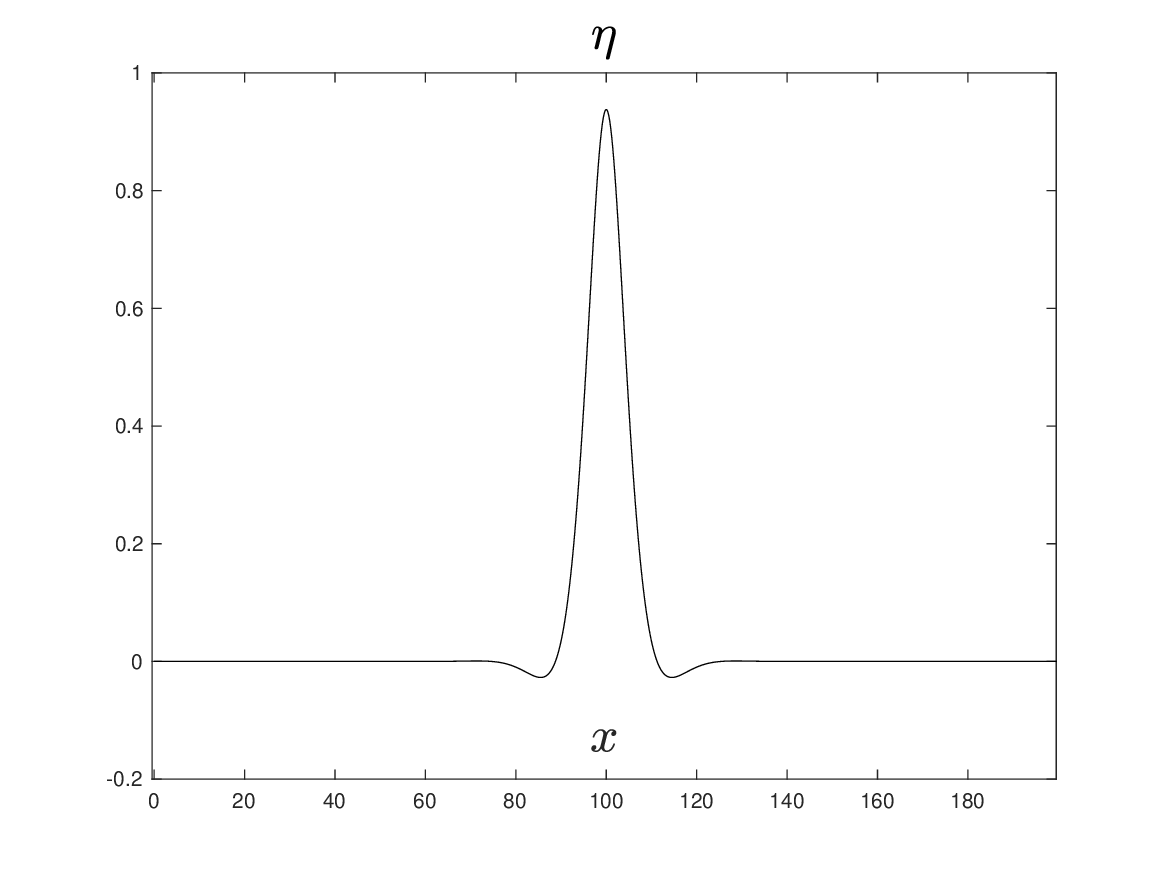}
         %\caption{}
         %\label{fig:three sin x}
     \end{subfigure}
        \caption{Solitary wave of the Boussinesq system \eqref{1bbl} computed with $b=d =2$, $b_2 = d_2=5$, $a=-2$, $c=-2$, $a_2=20$, $c_2=20$, $p=8$ and wave velocity $\omega = 0.8$.  }
        \label{Solit_u_eta1}
\end{figure}

\newpage

\begin{figure}[h!]
     \centering
     \begin{subfigure}[b]{0.45\textwidth}
         \centering
         \includegraphics[width=\textwidth, height=0.8\textwidth]{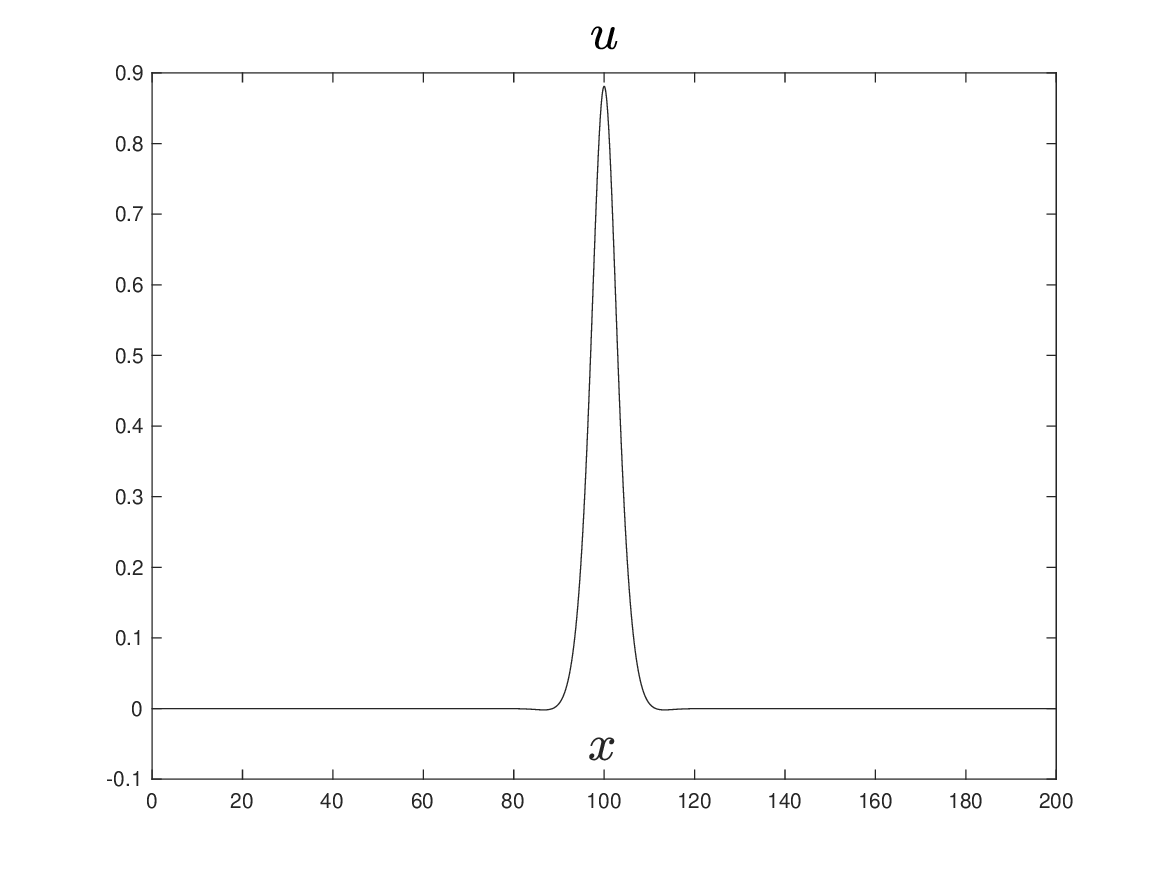}
         %\caption{}
         %\label{fig:y equals x}
     \end{subfigure}
     \begin{subfigure}[b]{0.45\textwidth}
         \centering
         \includegraphics[width=\textwidth, height=0.8\textwidth]{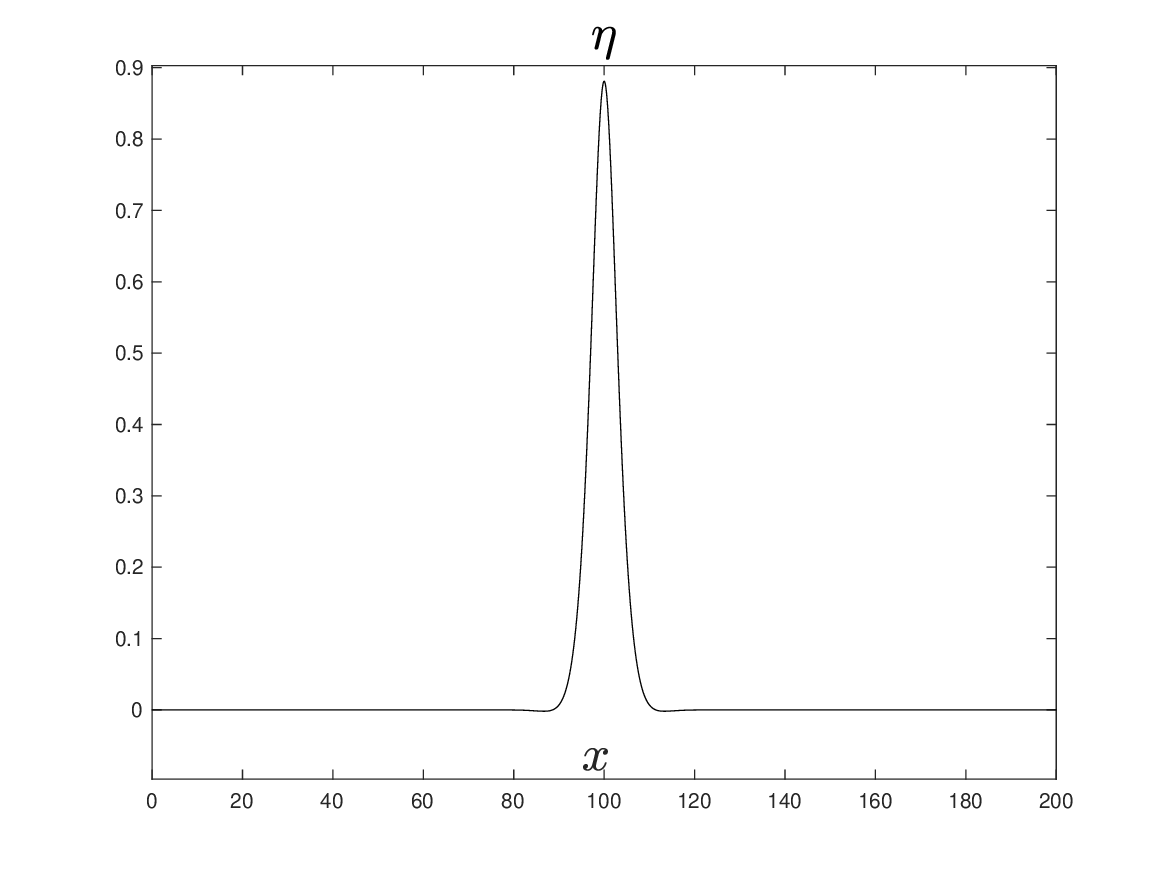}
         %\caption{}
         %\label{fig:three sin x}
     \end{subfigure}
        \caption{Solitary wave of the Boussinesq system \eqref{1bbl} computed with $b=d =4$, $b_2 = d_2=2$, $a=-4$, $c=-4$, $a_2=4$, $c_2=4$, $p=5$ and wave velocity $\omega = 0.8$.  }
        \label{Solit_u_eta2}
\end{figure}

\begin{figure}[h!]
     \centering
     \begin{subfigure}[b]{0.45\textwidth}
         \centering
         \includegraphics[width=\textwidth, height=0.8\textwidth]{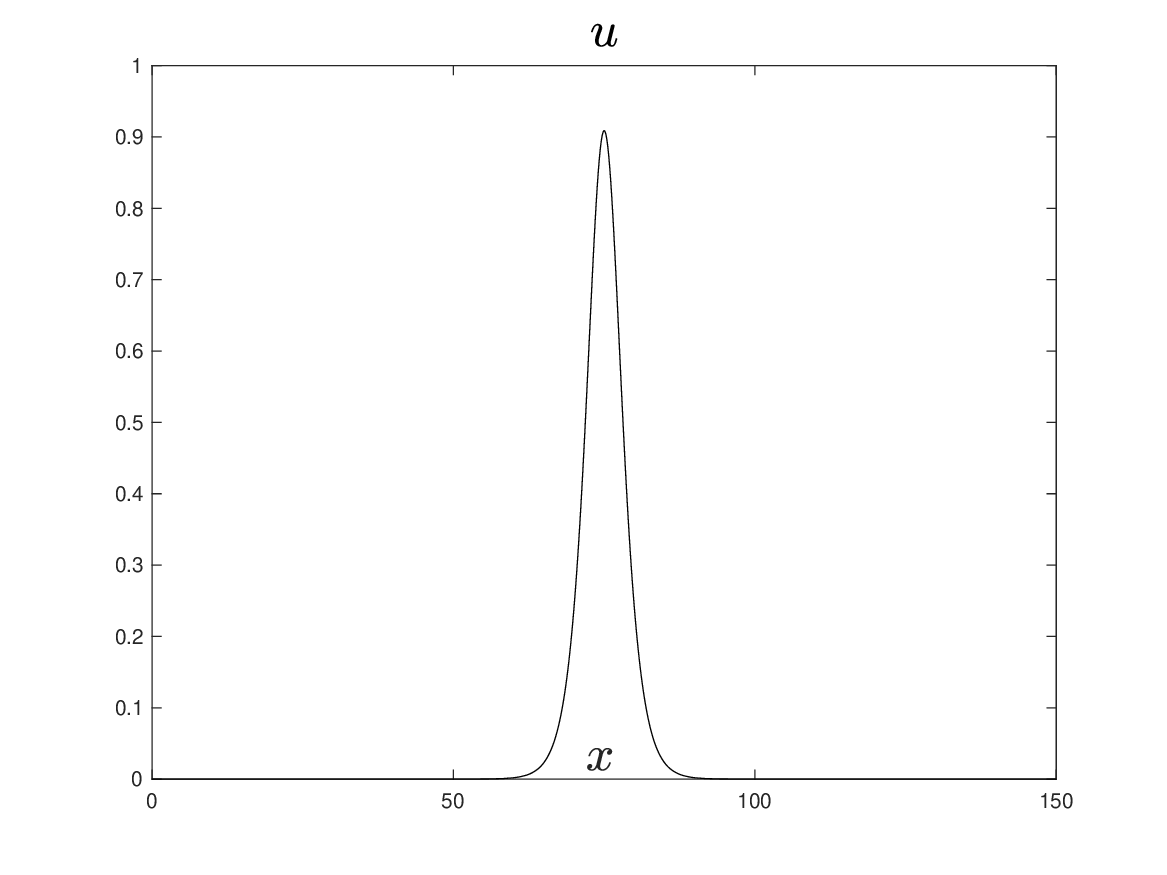}
         %\caption{}
         %\label{fig:y equals x}
     \end{subfigure}
     \begin{subfigure}[b]{0.45\textwidth}
         \centering
         \includegraphics[width=\textwidth, height=0.8\textwidth]{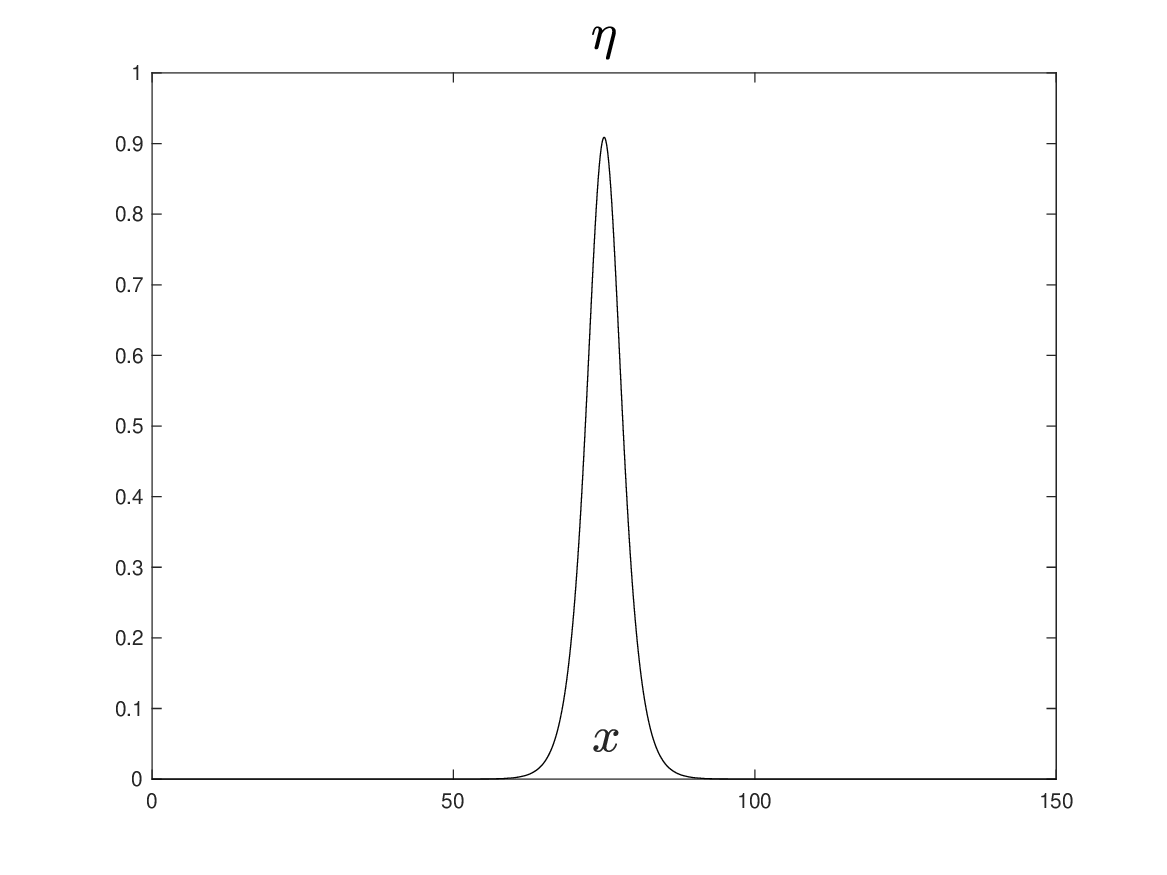}
         %\caption{}
         %\label{fig:three sin x}
     \end{subfigure}
        \caption{Solitary wave of the Boussinesq system \eqref{1bbl} computed with $b=d =4$, $b_2 = d_2=2$, $a=-4$, $c=-4$, $a_2=0.5$, $c_2=0.5$, $p=1$ and wave velocity $\omega = 0.4$.  }
        \label{Solit_u_eta3}
\end{figure}

Figure \ref{Solit_u_eta4} presents another experiment of this type with different model parameters, where we compute a solitary wave with velocity $\omega = 0.4$. In this case, the theoretical existence interval for solitary wave velocities is $0 < |\omega| < \min\left\{1,\frac{-a}{b}, \frac{-c}{b}, \frac{a_2}{b_2},\frac{c_2}{b_2}\right\} = 0.33$. All numerical parameters are identical to those used in Figure \ref{Solit_u_eta3}.

\begin{figure}[h!]
     \centering
     \begin{subfigure}[b]{0.45\textwidth}
         \centering
         \includegraphics[width=\textwidth, height=0.8\textwidth]{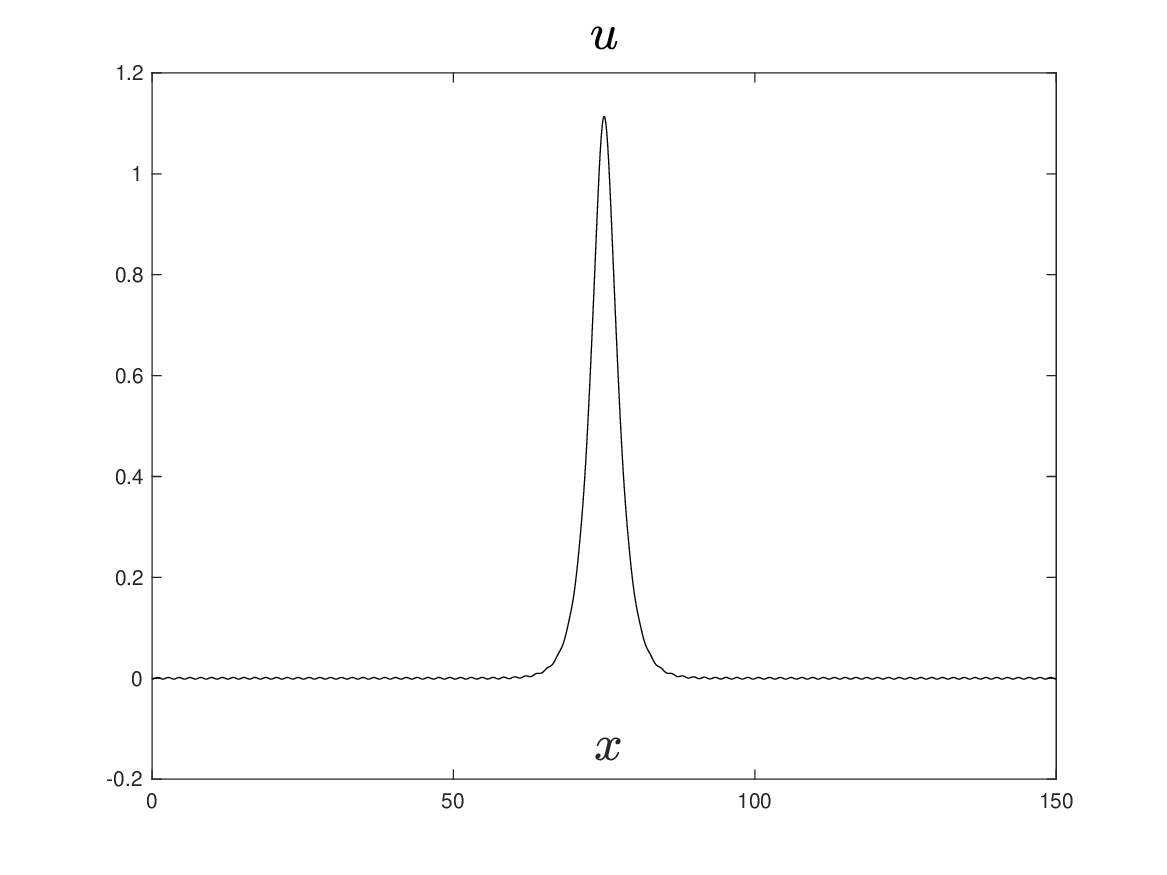}
         %\caption{}
         %\label{fig:y equals x}
     \end{subfigure}
     \begin{subfigure}[b]{0.45\textwidth}
         \centering
         \includegraphics[width=\textwidth, height=0.8\textwidth]{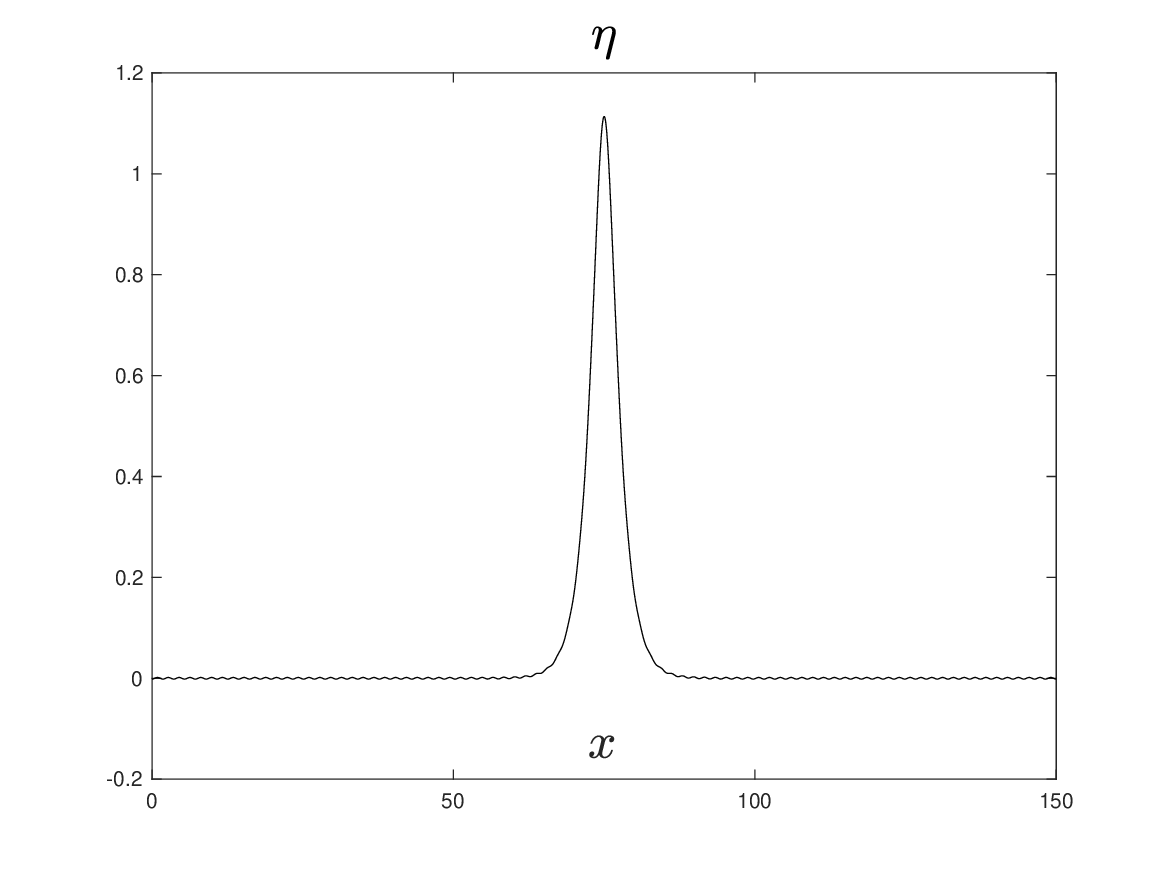}
         %\caption{}
         %\label{fig:three sin x}
     \end{subfigure}
        \caption{Solitary wave of the Boussinesq system \eqref{1bbl} computed with $b=d =4$, $b_2 = d_2=3$, $a=-4$, $c=-4$, $a_2=1$, $c_2=1$, $p=2$ and wave velocity $\omega = 0.4$.  }
        \label{Solit_u_eta4}
\end{figure}

\newpage
These numerical experiments indicate that solitary wave solutions of the Boussinesq system \eqref{1bbl} may exist beyond the previously established theoretical interval. However, the exact velocity range for the existence of solitary waves seems to be influenced by additional factors, such as the nonlinear exponent $p$.
%\newpage
\subsubsection{Verifying the approximations of solitary wave solutions} We now aim to verify the approximations of the solitary wave solutions presented in Figures \ref{Solit_u_eta1} and \ref{Solit_u_eta2}. To do this, we will compute the propagation of a solution in the one-dimensional fifth-order Boussinesq system \eqref{1bbl}. First, we apply the Fourier transform with respect to the spatial variable $x$ in this system, resulting in
\begin{align*}
\begin{cases}
&(1 + d k^2 + d_2 k^4) \hat{u}_t + ik \hat{\eta} + c (ik)^3 \hat{\eta} + c_2 (ik)^5 \hat{\eta} = ik \widehat{H}_1, \\
&(1 + b k^2 + b_2 k^4) \hat{\eta}_t + ik \hat{u} + a (ik)^3 \hat{u} + a_2 (ik)^5 \hat{u} = ik \widehat{H}_2,
\end{cases}
\end{align*}
which can be rewritten as
\begin{align}
\hat{u}_t = W_1 \hat{\eta} + \mathcal{L}_1, \label{u_equation}
\end{align}
and
\begin{align}
\hat{\eta}_t = W_2 \hat{u} + \mathcal{L}_2 \label{eta_equation},
\end{align}
where
\begin{align*}
&W_1 = \frac{ -i k + c i k^3 - c_2 i k^5}{ 1 + d k^2 + d_2 k^4 }, \quad W_2 = \frac{ -i k + a i k^3 - a_2 i k^5}{ 1 + b k^2 + b_2 k^4},\\
&\mathcal{L}_1 = \frac{i k \widehat{H}_1}{1 + d k^2 + d_2 k^4}, \quad \text{and} \quad \mathcal{L}_2 = \frac{i k \widehat{H}_2}{1 + b k^2 + b_2 k^4}.
\end{align*}
To approximate the time evolution of solutions to the full Boussinesq system \eqref{u_equation}-\eqref{eta_equation} numerically, we utilize the following \textit{time-stepping scheme}:
\begin{align}
&\frac{ \hat{u}^{n+1} - \hat{u}^n}{\Delta t} = W_1 ( \theta \hat{\eta}^{n+1} + (1-\theta)\hat{\eta}^n ) + \mathcal{L}_1^n, \label{theta_u_eq}
\end{align}
and
\begin{align}
&\frac{\hat{\eta}^{n+1} - \hat{\eta}^n }{\Delta t} = W_2 ( \theta \hat{u}^{n+1} + (1-\theta) \hat{u}^n ) + \mathcal{L}_2^n,\label{theta_eta_eq}
\end{align}
where $\theta \in [0,1]$ is a parameter. Now, by solving for $\hat{u}^{n+1}$ from equation \eqref{theta_u_eq}, and substituting it into equation \eqref{theta_eta_eq}, we obtain the system
\begin{align}
&\hat{u}^{n+1} = \frac{(1 + \Delta t^2 W_1 W_2 \theta (1-\theta)  ) \hat{u}^n + (\Delta t W_1 (1-\theta) + \Delta t \theta W_1 )\hat{\eta}^n + \Delta t^2 W_1 \theta \mathcal{L}_2^n + \Delta t \mathcal{L}_1^n   }{1 -\Delta t^2 W_1 W_2 \theta^2 }, \label{evol1}
\end{align}
and
\begin{align}
&\hat{\eta}^{n+1} = \hat{\eta}^n + \Delta t \theta W_2 \hat{u}^{n+1} + \Delta t (1-\theta) W_2 \hat{u}^n + \Delta t \mathcal{L}_2^n. \label{evol2}
\end{align}
Here, the superscript $n$ indicates that the corresponding quantity is evaluated at time $t = n \Delta t$. 

The quantities above allow us to assess the approximations of the solitary wave solutions depicted in Figures \ref{Solit_u_eta1} and \ref{Solit_u_eta2}. Specifically, Figures \ref{Solit_u_eta1_t=10} and \ref{Solit_u_eta2_t=10}, using the numerical scheme \eqref{evol1}-\eqref{evol2}, confirm the approximations of the solitary wave solutions shown in Figures \ref{Solit_u_eta1} and \ref{Solit_u_eta2}, respectively. The numerical parameters used are $N=2^{12}$ FFT points over the computational interval $[0,L]=200$, $\Delta t = 10/10000 = 0.001$ is the time step size. Note that in both experiments, the profiles of $u$ and $\eta$ propagate to the right with approximately constant shape and velocity $\omega = 0.8$, as expected. Neither dissipation nor dispersion is observed during wave propagation. These computer simulations provide evidence that the solutions computed using the scheme \eqref{iteration_3}-\eqref{iteration_4} behave as solitary wave solutions to the Boussinesq system \eqref{1bbl}, indicating that the numerical method performs quite well.

\begin{figure}[h!]
     \centering
     \begin{subfigure}[b]{0.45\textwidth}
         \centering
         \includegraphics[width=\textwidth, height=0.8\textwidth]{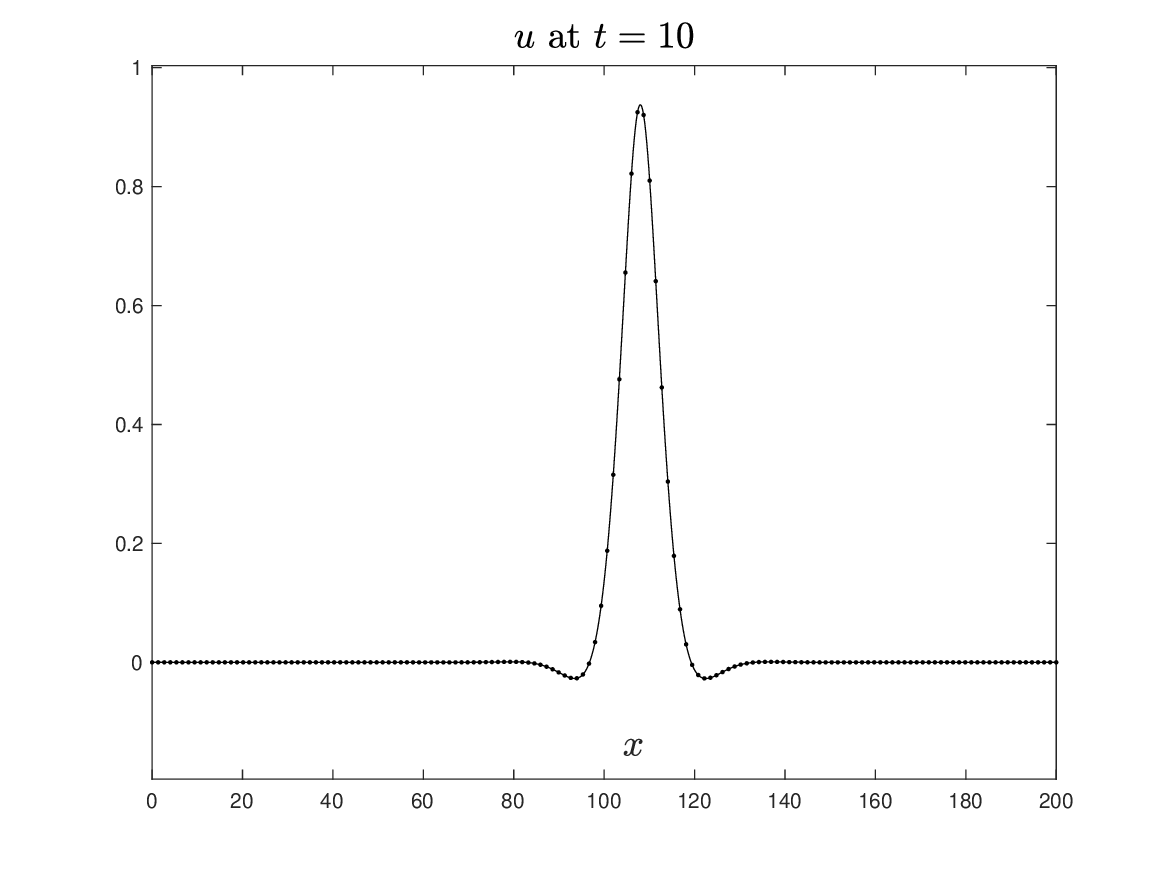}
         %\caption{}
         %\label{fig:y equals x}
     \end{subfigure}
     \begin{subfigure}[b]{0.45\textwidth}
         \centering
         \includegraphics[width=\textwidth, height=0.8\textwidth]{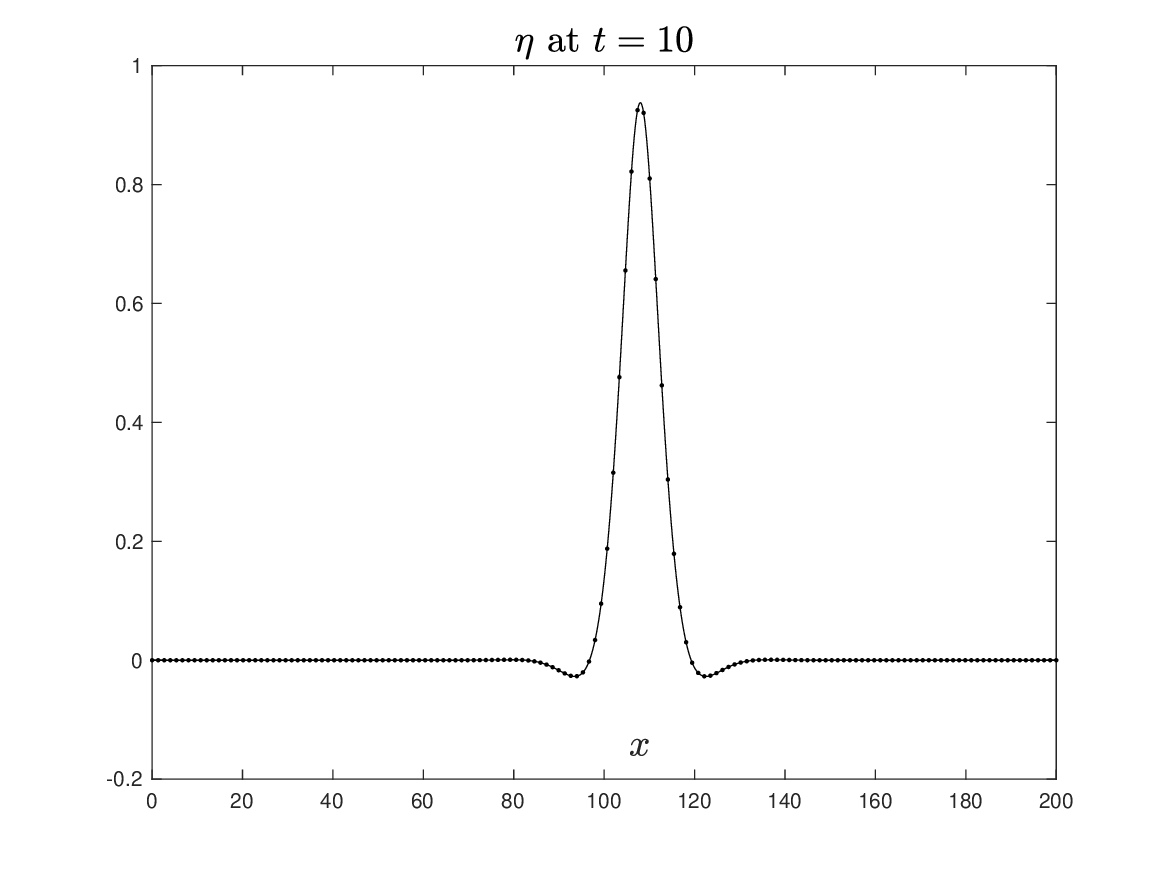}
         %\caption{}
         %\label{fig:three sin x}
     \end{subfigure}
        \caption{Solitary wave solution $(u, \eta)$ of the Boussinesq system \eqref{1bbl} computed at $t=10$, with parameters $b=d =2$, $b_2 = d_2=5$, $a=-2$, $c=-2$, $a_2=20$, $c_2=20$, $p=8$ and wave velocity $\omega = 0.8$. Dotted line: numerical solution computed using the scheme in \eqref{evol1}-\eqref{evol2}, with the initial condition given in Figure \ref{Solit_u_eta1}. Solid line: Approximate solitary wave shown in Figure \ref{Solit_u_eta1}, translated to the right by $\omega t = 8$. }
        \label{Solit_u_eta1_t=10}
\end{figure}

\begin{figure}[h!]
     \centering
     \begin{subfigure}[b]{0.45\textwidth}
         \centering
         \includegraphics[width=\textwidth, height=0.8\textwidth]{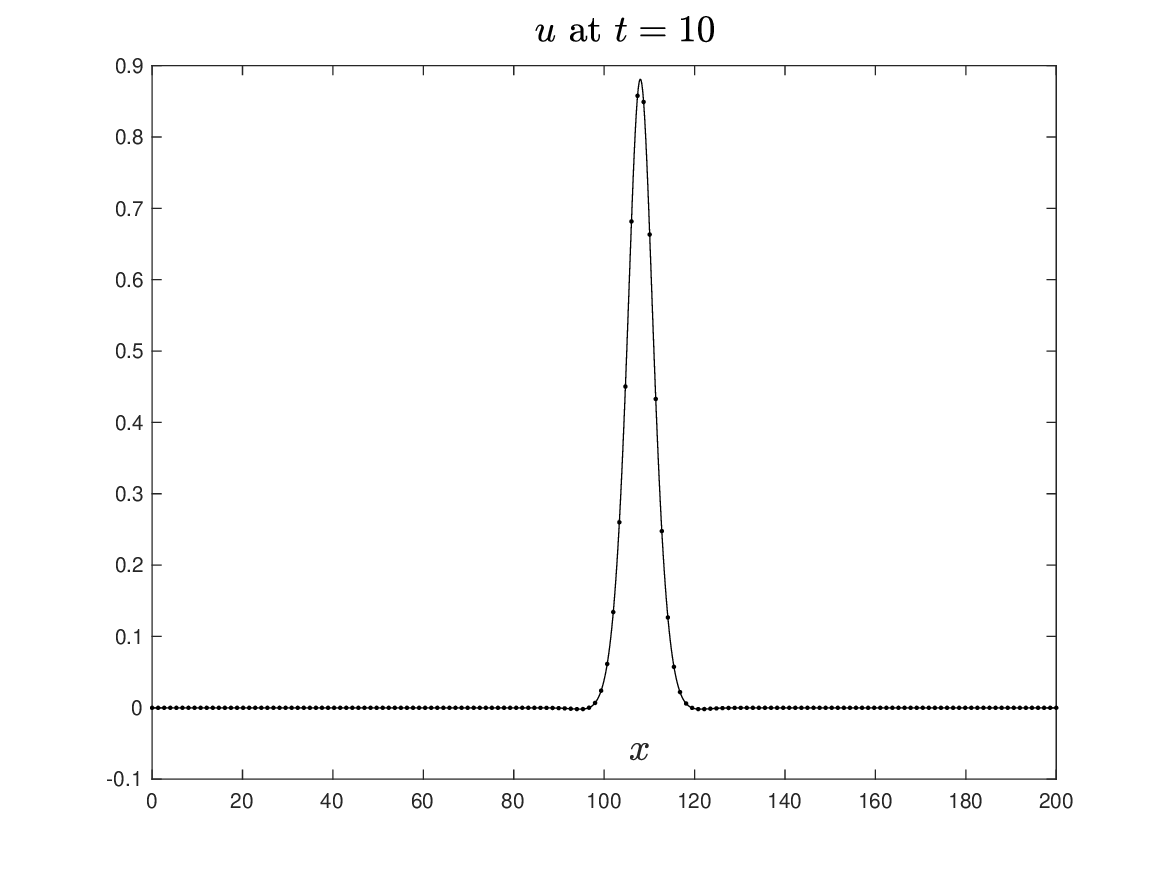}
         %\caption{}
         %\label{fig:y equals x}
     \end{subfigure}
     \begin{subfigure}[b]{0.45\textwidth}
         \centering
         \includegraphics[width=\textwidth, height=0.8\textwidth]{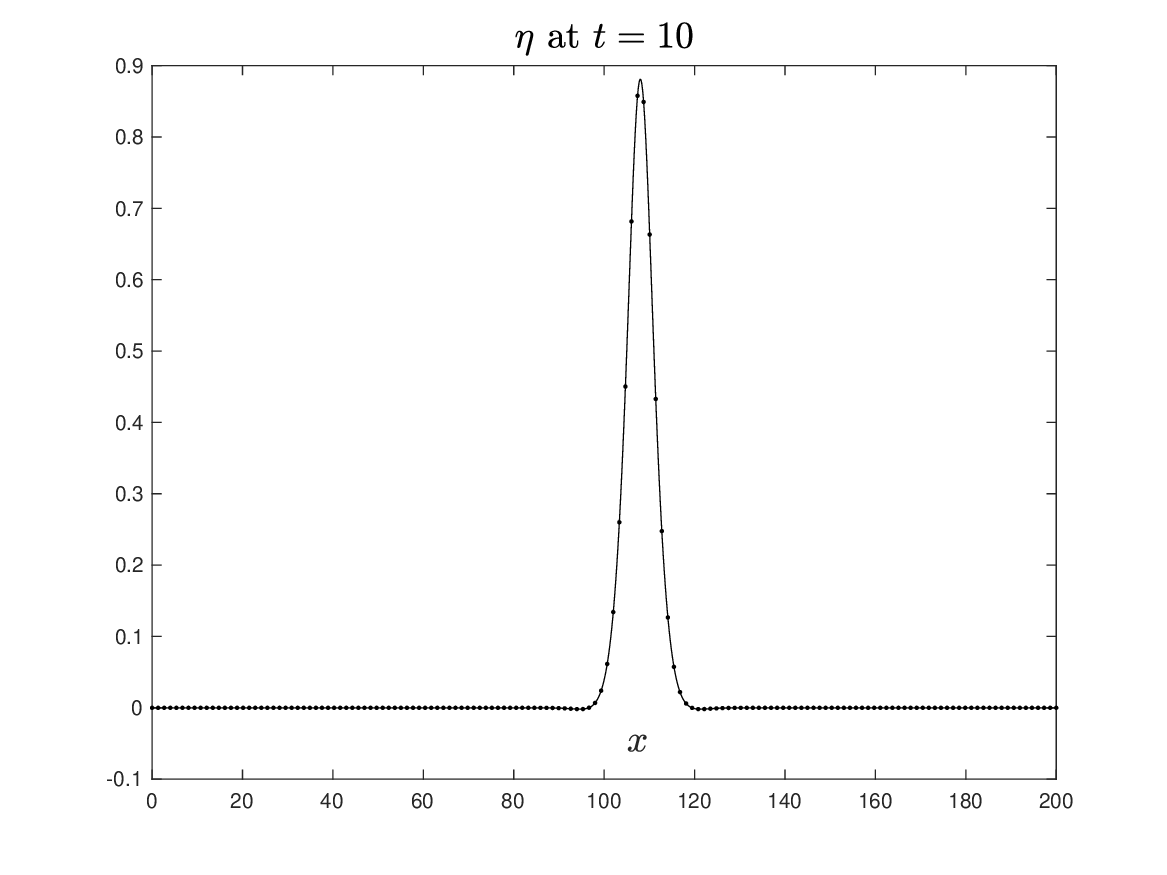}
         %\caption{}
         %\label{fig:three sin x}
     \end{subfigure}
        \caption{Solitary wave solution $(u,\eta)$ of the Boussinesq system \eqref{1bbl} computed at $t=10$, with parameters $b=d =2$, $b_2 = d_2=2$, $a=-4$, $c=-4$, $a_2=4$, $c_2=4$, $p=5$ and wave velocity $\omega = 0.8$. Dotted line: numerical solution computed using the scheme in \eqref{evol1}-\eqref{evol2}, with the initial condition given in Figure \ref{Solit_u_eta2}. Solid line: Approximate solitary wave shown in Figure \ref{Solit_u_eta2}, translated to the right by $\omega t = 8$. }
        \label{Solit_u_eta2_t=10}
\end{figure}

\newpage
\subsection{Non-homogeneous case}

In this set of computer simulations, we aim to explore the case where the nonlinearities $H_1, H_2$ are non-homogeneous. Specifically, we consider
\begin{equation*}
F(\eta, \partial_x\eta, u, \partial_xu) = \frac14 u^4  + u (\partial_x u)^2 + \frac14 \eta^4 + \eta (\partial_x \eta)^2.
\end{equation*}
%\begin{equation*}
%F(\eta, \partial_x\eta, u, \partial_xu) = \frac14 u^4 + %\frac12\partial_x(u^2) + \frac14 \eta^4 + \frac{1}{2}\partial_x(\eta^2).
%\end{equation*}
By substituting into the expressions for the functions $H_1, H_2$, we obtain
\begin{align*}
&H_1(\eta, \partial_x\eta, \partial^2_x\eta, u, \partial_xu,\partial^2_x u) = \eta^3 - (\partial_x\eta)^2 - 2(\partial_x^2\eta)\eta,
\end{align*}
and
\begin{align*}
&H_2(\eta, \partial_x\eta, \partial^2_x\eta, u, \partial_xu,\partial^2_x u) = u^3 - (\partial_xu)^2 - 2(\partial^2_x u) u.
\end{align*}
In this case, we use a different numerical scheme from the one applied in the homogeneous case, as the stabilizing factors cannot be directly computed.

To approximate the solitary wave equations \eqref{trav-eqs}, we
 expand the unknowns $\psi(x)$ and $v(x)$ functions with period $L=2l$ (for sufficiently large $l$), using the following cosine expansions:
\begin{equation*}\label{cosine_expansions}
\begin{cases}
&\displaystyle\psi(x) = \psi_0 + \sum_{k=1}^{N/2} \psi_k \cos \Big( \frac{k\pi x}{l} \Big),\\
\\
&\displaystyle v(x) = v_0 + \sum_{k=1}^{N/2} v_k \cos \Big( \frac{k \pi x}{l} \Big).
\end{cases}
\end{equation*}
Substituting these expressions into the solitary wave equations \eqref{trav-eqs}, and evaluating at the $N/2 +1$ collocation points
$$
x_j = \frac{2l (j-1)}{N}, ~~ j =1,..., N/2 +1,
$$
we obtain a system of $N + 2$ nonlinear equations of the form
$$
F(\psi_0, \psi_1,..., \psi_{N/2}, v_0, v_1,..., v_{N/2}) = 0,
$$
where the unknowns are the coefficients $\psi_k, v_k$. This nonlinear system is solved using Newton's iteration. The iteration process is stopped when the relative difference between two successive approximations and the value of the field $F:\mathbb{R}^{N+2} \to \mathbb{R}^{N+2}$ is less than $10^{-12}$.
The initial guess for the Newton iteration is taken as
\[
v^0(x) = \psi^0(x) = e^{-0.05(x-a_0)^2},
\]
with $a_0=50$.
The result of this numerical experiment is presented in Figure \ref{Solit_u_eta5} using $N = 2^{12}$ FFT points over the computational domain $[0,100]$.

\begin{figure}[h!]
     \centering
     \begin{subfigure}[b]{0.45\textwidth}
         \centering
         \includegraphics[width=\textwidth, height=0.8\textwidth]{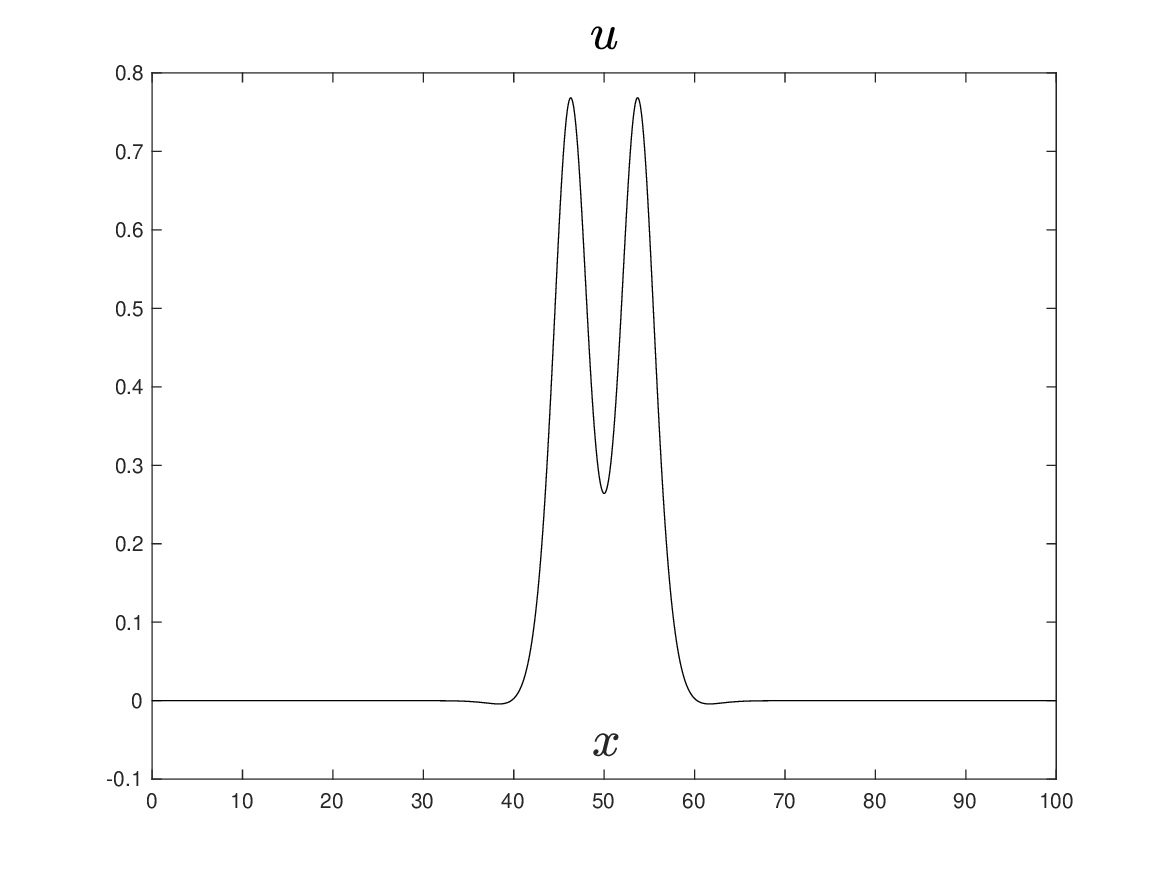}
         %\caption{}
         %\label{fig:y equals x}
     \end{subfigure}
     \begin{subfigure}[b]{0.45\textwidth}
         \centering
         \includegraphics[width=\textwidth, height=0.8\textwidth]{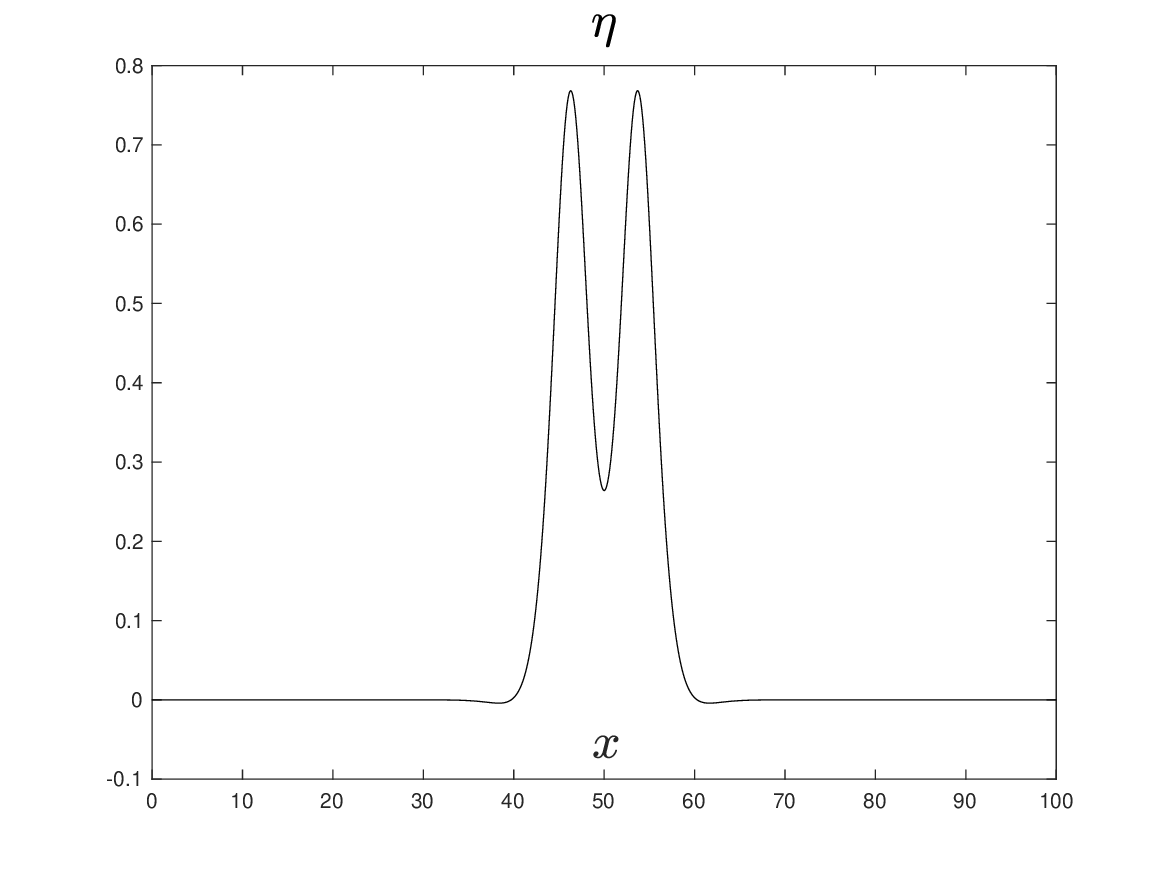}
         %\caption{}
         %\label{fig:three sin x}
     \end{subfigure}
        \caption{Solitary wave of the Boussinesq system \eqref{1bbl} computed with $b=d =2$, $b_2 = d_2=3$, $a=-2$, $c=-2$, $a_2=3$, $c_2=3$, and wave velocity $\omega = 0.6$.  }
        \label{Solit_u_eta5}
\end{figure}

%\newpage

In Figure \ref{Solit_u_eta5_t=10}, we apply the numerical scheme defined by equations \eqref{evol1}-\eqref{evol2} to validate the approximation of the solitary wave solution shown in Figure \ref{Solit_u_eta5}. The numerical setup uses $N=2^{12}$ FFT points over the computational domain $[0,L]=[0,100]$, with a time step size of $\Delta t = 10/10000 = 0.001$. As expected, the wave profiles propagate to the right, preserving their shape and moving at approximately the expected velocity.

\begin{figure}[h!]
     \centering
     \begin{subfigure}[b]{0.45\textwidth}
         \centering
         \includegraphics[width=\textwidth, height=0.8\textwidth]{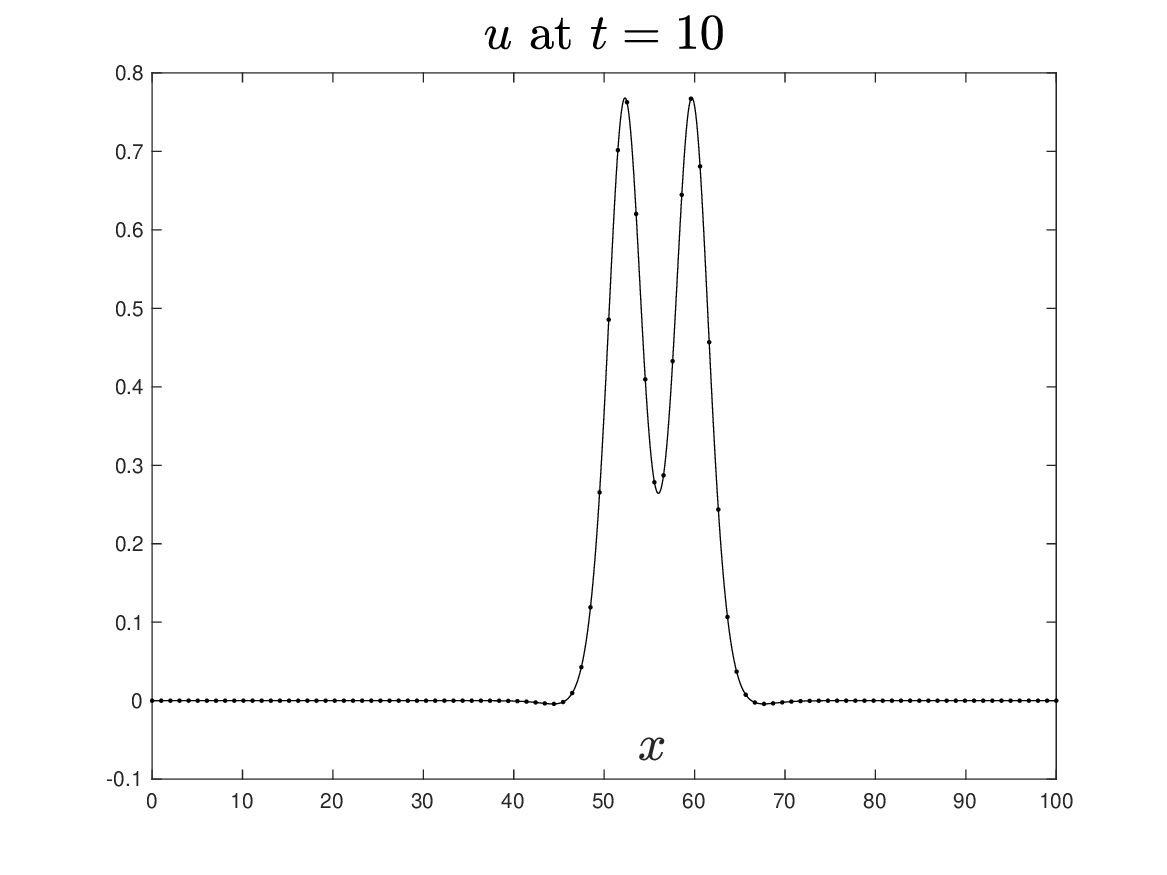}
         %\caption{}
         %\label{fig:y equals x}
     \end{subfigure}
     \begin{subfigure}[b]{0.45\textwidth}
         \centering
         \includegraphics[width=\textwidth, height=0.8\textwidth]{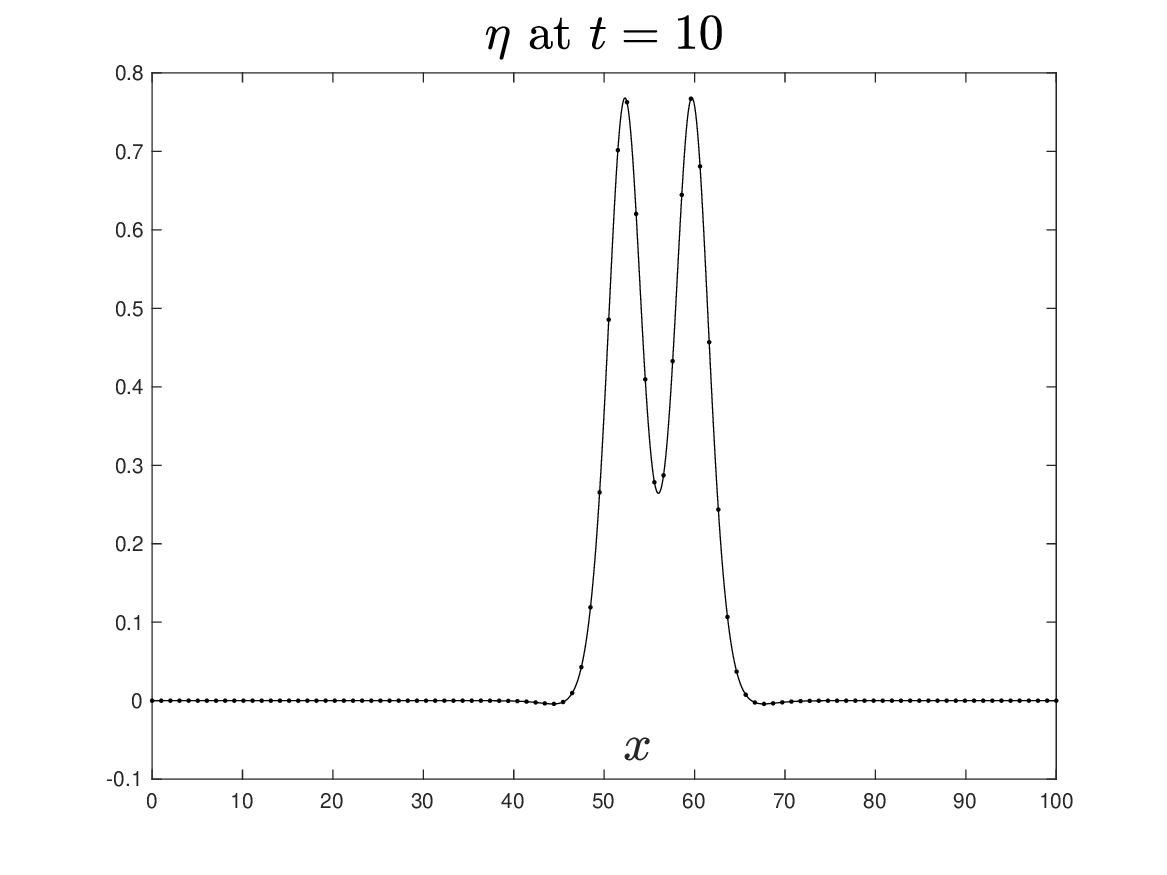}
         %\caption{}
         %\label{fig:three sin x}
     \end{subfigure}
        \caption{Solitary wave solution $(u,\eta)$ of the Boussinesq system \eqref{1bbl} computed at $t=10$, with parameters $b=d =2$, $b_2 = d_2=3$, $a=-2$, $c=-2$, $a_2=3$, $c_2=3$, and wave velocity $\omega = 0.6$. Dotted line: numerical solution computed using the scheme in \eqref{evol1}-\eqref{evol2}. Solid line: Approximate solitary wave shown in Figure \ref{Solit_u_eta5}. }
        \label{Solit_u_eta5_t=10}
\end{figure}

In Figure \ref{Solit_u_eta10}, a solitary wave solution is presented, computed using a different set of parameters satisfying $a \neq c$, $a_2 \neq c_2$. The numerical parameters remain the same as in the previous experiment. Unlike earlier cases, the height and shape of the component profiles $u$ and $\eta$ now differ noticeably.

\begin{figure}[h!]
     \centering
     \begin{subfigure}[b]{0.45\textwidth}
         \centering
         \includegraphics[width=\textwidth, height=0.8\textwidth]{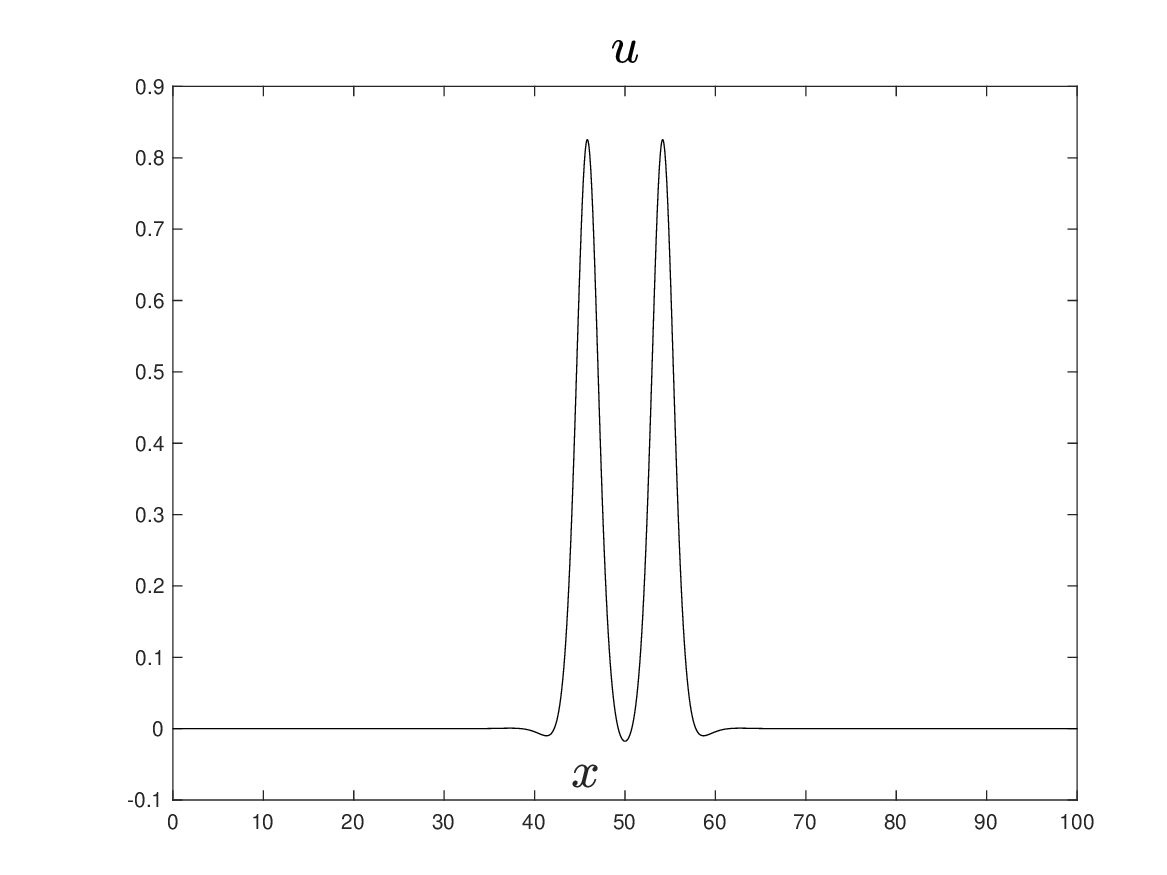}
         %\caption{}
         %\label{fig:y equals x}
     \end{subfigure}
     \begin{subfigure}[b]{0.45\textwidth}
         \centering
         \includegraphics[width=\textwidth, height=0.8\textwidth]{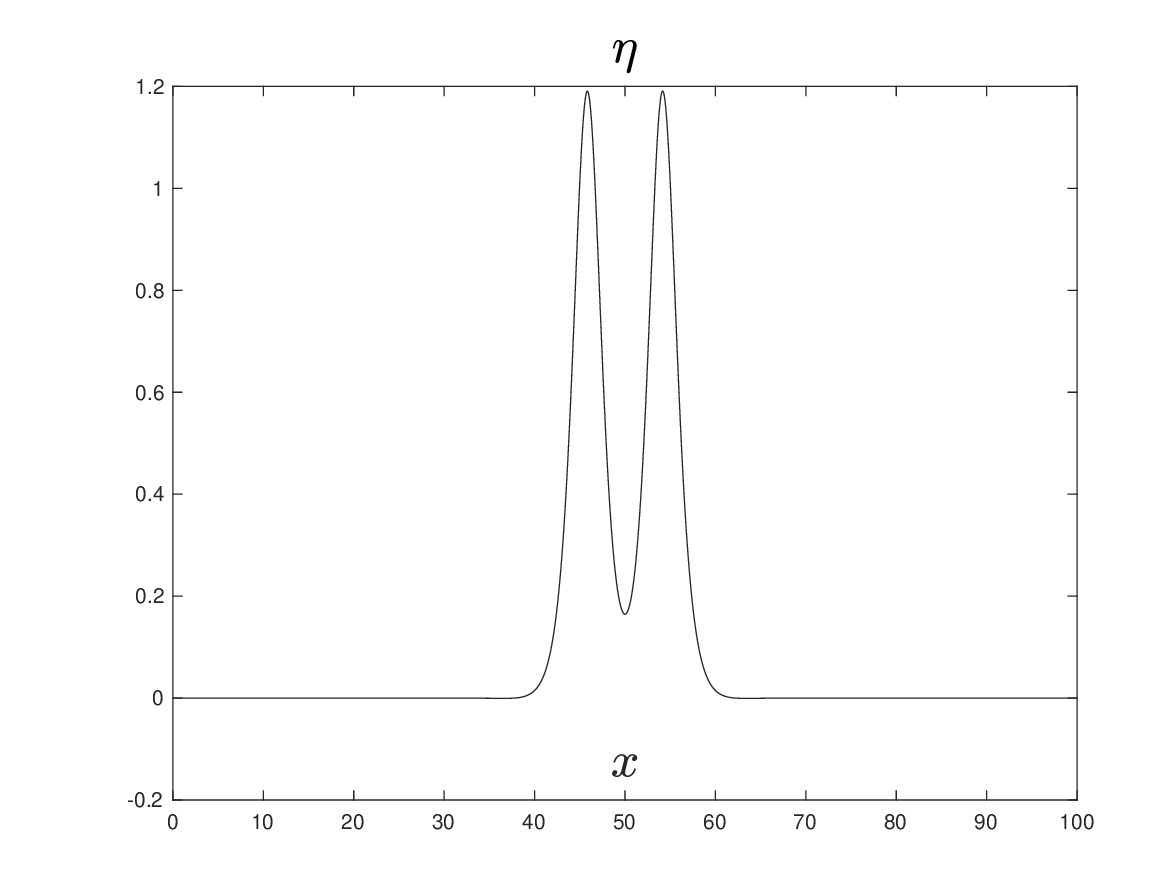}
         %\caption{}
         %\label{fig:three sin x}
     \end{subfigure}
        \caption{Solitary wave of the Boussinesq system \eqref{1bbl} computed with $b=d =2$, $b_2 = d_2 = 4$, $a=-1$, $c=-2$, $a_2=2$, $c_2=1$, and wave velocity $\omega = 0.2$. Note that $a \neq c$, $a_2 \neq c_2$.}
        \label{Solit_u_eta10}
\end{figure}

Finally, in Figure \ref{Solit_u_eta11}, we present a solitary wave solution computed with parameters satisfying $a \neq c$, $a_2 \neq c_2$ and using the initial guesses
\[
v^0(x) = e^{-0.05(x-50)^2}, ~~~~ \psi^0(x) = 1.5 e^{-0.1 (x-50)^2}.
\]
Note that in this case, $v^0(x) \neq \psi^0(x)$, highlighting the flexibility of the iterative solver in handling asymmetric initial data.
Moreover, as previously observed, the solution profiles of $u$ and $\eta$ remain asymmetric and display significantly different dynamics.

\begin{figure}[h!]
     \centering
     \begin{subfigure}[b]{0.45\textwidth}
         \centering
         \includegraphics[width=\textwidth, height=0.8\textwidth]{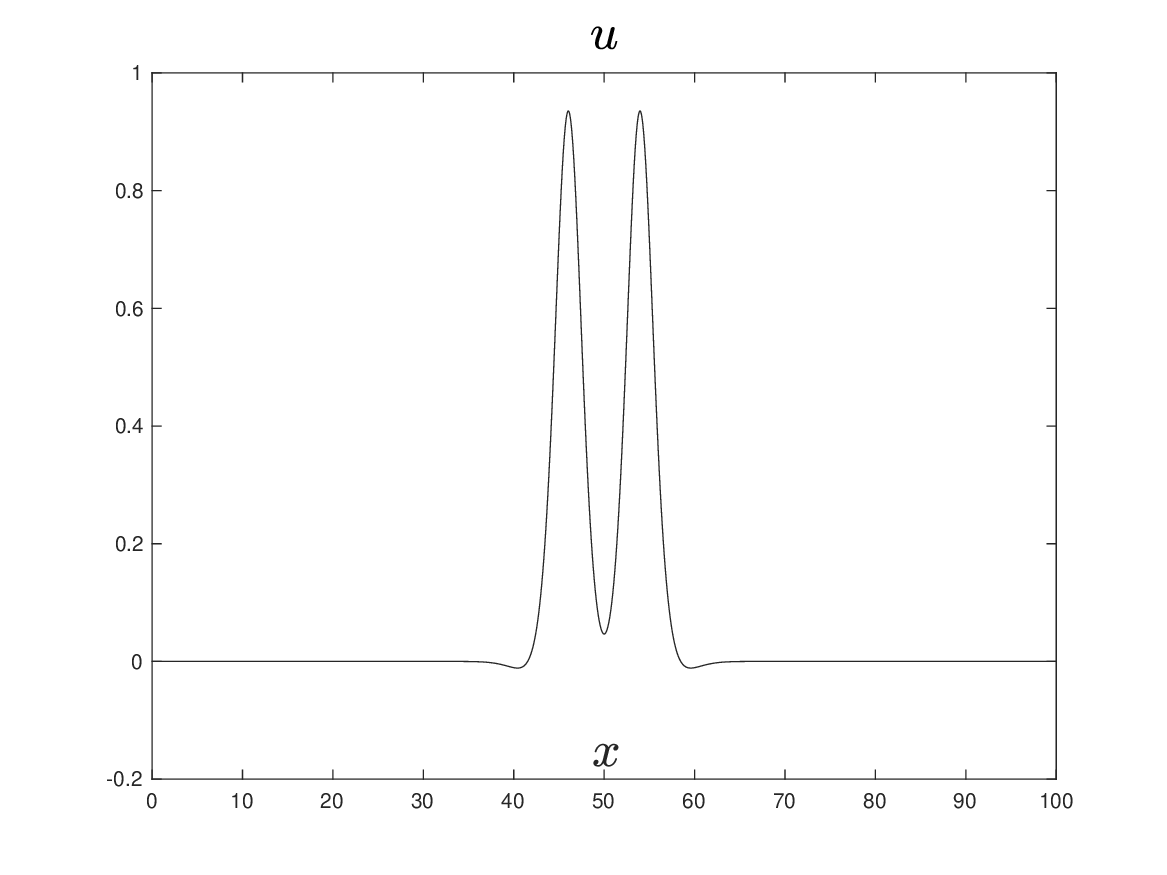}
         %\caption{}
         %\label{fig:y equals x}
     \end{subfigure}
     \begin{subfigure}[b]{0.45\textwidth}
         \centering
         \includegraphics[width=\textwidth, height=0.8\textwidth]{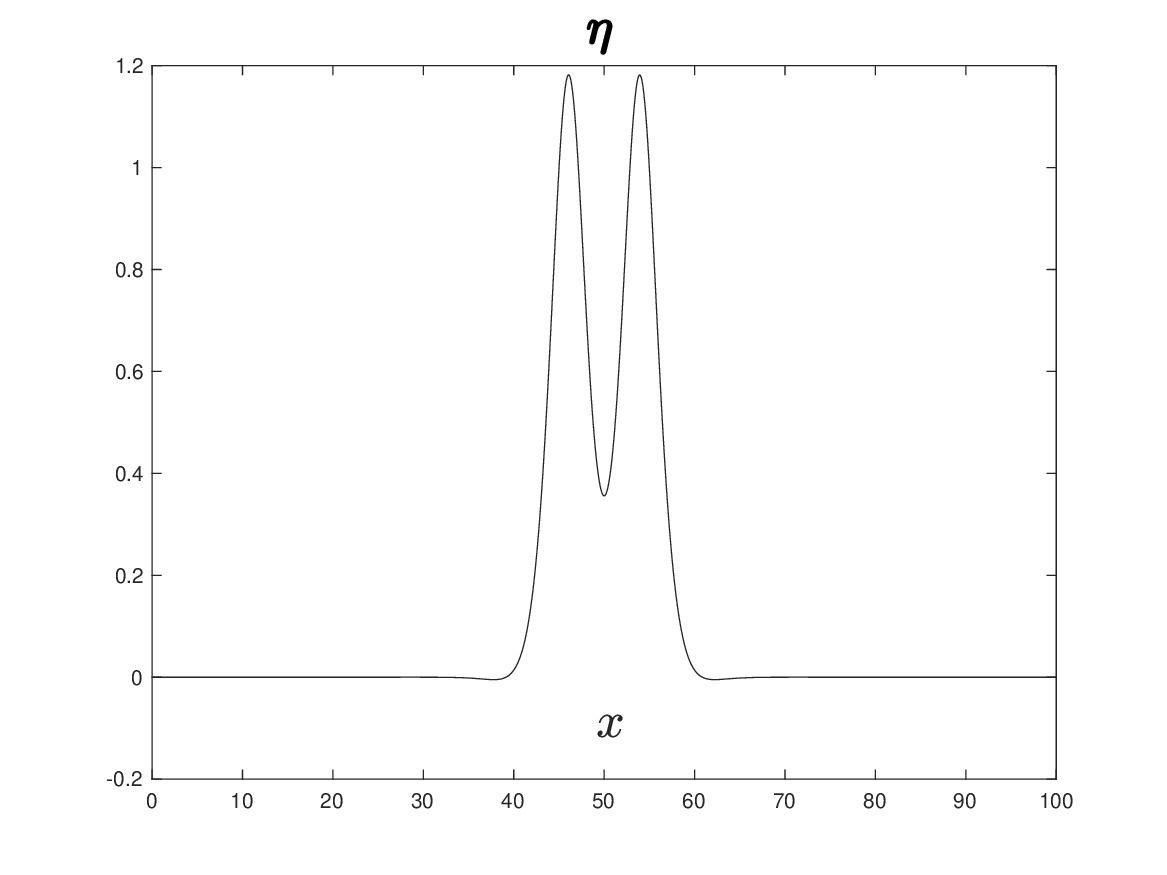}
         %\caption{}
         %\label{fig:three sin x}
     \end{subfigure}
        \caption{Solitary wave of the Boussinesq system \eqref{1bbl} computed with parameters $b=d=2$, $b_2 = d_2 = 2$, $a=-1$, $c=-2$, $a_2=1$, $c_2=3$, and wave velocity $\omega = 0.2$. Note that $a \neq c$, $a_2 \neq c_2$ and the initial guess for the iterative solver satisfies $v^0 \neq \psi^0$ in this experiment. }
        \label{Solit_u_eta11}
\end{figure}

%\section{Additional remarks and future directions}\label{sec4}

\section{Conclusions}
In this work, we have investigated the existence and numerical computation of traveling wave solutions for a general class of nonlinear higher-order Boussinesq evolution systems with a Hamiltonian structure. Using the system's variational formulation and applying the concentration-compactness principle of P.-L. Lions, we established the existence of finite-energy traveling wave solutions, even in the presence of non-homogeneous nonlinearities.

To complement the theoretical results, we developed and implemented spectral numerical methods applicable to both homogeneous and non-homogeneous cases. In the homogeneous setting, we employed a Fourier spectral method with an iterative scheme enhanced by stabilizing factors to ensure convergence. For non-homogeneous nonlinearities, we utilized a Fourier collocation method combined with Newton’s iteration. The numerical simulations validated the theoretical predictions and further provided insight into the behavior of the solution beyond the range covered by the analytical results.

Our findings highlight the influence of additional parameters, such as the specific form and degree of the nonlinearities, on the existence and shape of solitary traveling waves. In particular, we observed that non-monotonic and oscillatory wave profiles may emerge under certain parameter regimes, a phenomenon consistent with similar results reported in other nonlinear dispersive models.

Overall, this study contributes both theoretical and numerical tools for understanding solitary wave dynamics in higher-order Boussinesq systems and opens further research into more complex wave interactions and stability properties in such models.

%\newpage 
\subsection*{Acknowledgment} The authors are grateful to the referees for the careful reading of this paper and their valuable suggestions and comments.  R. de A. Capistrano–Filho was partially supported by CAPES/COFECUB grant number 88887.879175/2023-00, CNPq grant numbers 301744/2025-4, 421573/2023-6, and 307808/2021-1, and Propesqi (UFPE). J. C. Mu\~noz was supported by the Department of Mathematics of the Universidad del Valle (Colombia) under the C.I. 71360 project. J. Quintero would like to thank the Universidad del Valle for the sabbatical year (08/2024-08/2025).

% ------------------------------------------------------------------------

\begin{thebibliography}{1}

\bibitem{BCL-15} E. S. Bao, R. M. Chen, and Q. Liu, \textit{Existence and symmetry of ground states to the Boussinesq abcd systems}, Arch. Rat. Mech. Anal. \textbf{216}:2, 569--591 (2015).

\bibitem{Benney-1977} D. J. Benney, \textit{A general theory for interactions between short and long waves}, Stud. Appl. Math., \textbf{56}, 81--94 (1977).

\bibitem{BCS02} J. L. Bona, M. Chen, and J. C. Saut, \textit{Boussinesq equations and other systems for small-amplitude long waves in nonlinear dispersive media. I. Derivation and linear theory,} J. Nonlinear Sci., \textbf{12},  283--318 (2002).

\bibitem{BCS04}J. L. Bona, M. Chen, and J. C. Saut, \textit{Boussinesq equations and other systems for small-amplitude long waves in nonlinear dispersive media. II: The nonlinear theory, } Nonlinearity, \textbf{17}, 925--952 (2004).

\bibitem{BCL2005} J. L. Bona, T. Colin, and D. Lannes, \textit{Long wave approximations for water waves}, Arch. Ration. Mech. Anal. \textbf{178}:3, 373--410 (2005).

\bibitem{Dougalis} J. L. Bona, V. A. Dougalis, and D. E. Mitsotakis, \textit{Numerical solution of Boussinesq systems of KdV-KdV type: II Evolution of radiating solitary waves}, Nonlinearity \textbf{21}, no. 12, 2825--2848 (2008).

\bibitem{CQS-2024} R. de A. Capistrano-Filho, J. Quintero, and S.-M. Sun, \textit{Orbital stability of the $abcd$-Boussinesq system: The Hamiltonian case}, Nonlinearity, \textbf{38} 075034, 1-41 (2025). 

\bibitem{CNS2011} M. Chen, N. V. Nguyen, and S.-M. Sun,  \textit{Existence of traveling-wave solutions to Boussinesq systems}, Differ. Integral Equ. \textbf{24} (9--10), 895--908 (2011).

\bibitem{Craig-Groves-1994} W. Craig and M. D. Groves, \textit{Hamiltonian long-wave approximations to the water-wave problem}, Wave Motion, \textbf{19}, 367--389 (1994).

\bibitem{Duran} V. A. Dougalis, A. Dur\'an and L. Saridaki, \textit{On solitary-wave solutions of Boussinesq/Boussinesq systems for internal waves}, Physica D: Nonlinear Phenomena \textbf{428}, 133051 (2021).

\bibitem{Dursun} I. Dursun and D. Idris, \textit{Numerical simulations of the improved Boussinesq equation}, Num. Met. Partial Diff. Eqs. \textbf{26}, No. 6, 1316--1327 (2009).

\bibitem{Ersoy} O. Ersoy, A. Korkmaz, and I. Dag, \textit{Exponential B-Splines for numerical solutions to some Boussinesq systems for water waves}, Mediterranean J. Math. \textbf{13}, 4975--4994 (2016). 

\bibitem{Esfahani-Levandosky-2021} A. Esfahani and S. Levandosky, \textit{Existence and stability of traveling waves of the fifth-order KdV equation}, Physica D, \textbf{421}, 132872 (2021).

\bibitem{HTS-13} S. Hakkaev, M. Stanislavova, and A. Stefanov, \textit{Spectral stability for subsonic traveling pulses of the Boussinesq ``abc” system}. SIAM J. Appl. Dyn. Syst., \textbf{12}:2, 878--898 (2013).

\bibitem{Hunter-Scheurle-1988} J. Hunter and J. Scheurle, \textit{Existence of perturbed solitary wave solutions to a model equation for water waves}, Physica D, \textbf{32}, (1988) 253--268.


\bibitem{Kawahara-1972} T. Kawahara, \textit{Oscillatory solitary waves in dispersive media}, Journal of the Physical Society of Japan, \textbf{33}, 260--264 (1972).

\bibitem{Kichenassamy-Olver-1991} S. Kichenassamy and P. Olver, \textit{Existence and non-existence of solitary waves to higher-order evolution equations}, IMA Preprint Series, \textbf{803} (1991).


\bibitem{KS-24} C. Klein and J-C. Saut, \textit{Numerical study of the Amick–Schonbek system}, Stud Appl Math., \textbf{153}:e12691, (2024).

\bibitem{KS-25} C. Klein and J-C. Saut, \textit{On the Kaup–Broer–Kupershmidt systems}, EMS Surv. Math. Sci., \textbf{12}:1, 215--242 (2025).

\bibitem{Marinov}  T. Marinov and R.S. Marinova, {\it Solitary wave solutions with non-monotone shapes for the modified Kawahara equation}, J. Comput. Appl. Math, \textbf{340} (4), (2017). 

\bibitem{Lev-2007} S. Levandosky, {\it Stability of solitary waves of a fifth-order water wave model}, Physica D: Nonlinear Phenomena, \textbf{227}:2 (2007).

\bibitem{Lions-1984a} P.-L. Lions, \textit{The concentration-compactness principle in the calculus of variations. The locally compact case. Part I},  Annales de l'Institut Henri Poincaré, Analyse Non Linéaire, \textbf{1}, 109--145 (1984).

\bibitem{Lions-1984b} P.-L. Lions, \textit{The concentration-compactness principle in the calculus of variations. The locally compact case. Part II},  Annales de l'Institut Henri Poincaré, Analyse Non Linéaire, \textbf{1}, 223--283 (1984).

\bibitem{Olver-1984} P. J. Olver, Hamiltonian and non-Hamiltonian models for water waves. In \textit{Lecture Notes in Physics}, Springer-Verlag, \textbf{195}, 273--290.

%\bibitem{Pava} J. Angulo Pava, On the instability of solitary-wave solutions for fifth-order water wave models, Electron. J. Differential Equations (6) (2003).

\bibitem{Ponce-1993} G. Ponce, \textit{Lax pairs and higher-order models for water waves},  Journal of Differential Equations, \textbf{102} (2), 360--381 (1993) .

\bibitem{Zufiria-1987} J. A. Zufiria, \textit{Symmetry breaking in periodic and solitary gravity-capillary waves on water of finite depth}, Journal of Fluid Mechanics, \textbf{184}, 183--206 (1987).



\bibitem{GSS}  M. Grillakis, J. Shatah, W. Strauss,  \textit{Stability Theory of Solitary Waves in Presence of Symmetry, I.} Functional Anal. \textbf{74}, 160-197  (1987).

%\bibitem{GSS1}  M. Grillakis, J. Shatah, and W. Strauss,  \textit{Stability Theory of Solitary Waves in Presence of Symmetry, II.} Functional Anal., 94, 308--348 (1990).


%\bibitem{JQAM2013c} A. M. Montes, J. R. Quintero, \textit{On the nonlinear scattering and well-posedness for a 2D-Boussinesq-Benney-Luke type system.}  Journal of Differential Equations, \textbf{260 (7)} (2015), 6057-6083.

%\bibitem{JQ2002a}  J. R. Quintero, \textit{Existence and analyticity of lump solutions for
%generalized Benney-Luke equations}.  Rev. Col. Mat.  \textbf{36}
%(2002), 71-95.

%\bibitem{Q1} J. R. Quintero,
%\textit{Nonlinear stability of a one-dimensional Boussinesq
%equation}. Dynam. Diff. Eq. \textbf{15} (2003), 125-142.

%\bibitem{JQ2005} J. R. Quintero, \textit{Nonlinear stability of solitary
%waves for a 2-D Benney-Luke equation}. Discrete Contin. Dynam.
%Systems \textbf{13} (2005), 203-218.

%\bibitem{JQ2010b} J. R. Quintero, \textit{The Cauchy problem and stability of solitary waves for a 2D Boussinesq-KdV type system}. Diff. and Int. Eq. \textbf{23} (2010),
%325-360.


%\bibitem{BSS} J. Bona, P. Souganidis, W. Strauss, \textit{Stability and instability of solitary waves of Korteweg de
%Vries type equation}, Proc. Roy. Soc. London. Ser A,  411:1841, 395--412 (1987).

%\bibitem{BC2002} J.L. Bona, and H. Chen, \textit{Solitary waves in nonlinear dispersive systems,} Discrete and Continuous Dynamical Systems, 2:3, 313--378 (2002).


%\bibitem{DBJCb} A. de Bouard and J. C. Saut, \textit{Solitary waves of generalized Kadomtsev-Petviashvili equations},  Ann. Inst. H. Poincar\'e Anal. Non Lin\'eaire, 14:2, 211--236 (1997).


%\bibitem{YLXP} Y. Liu, and X-P. Wang, \textit{Nonlinear stability of solitary waves of a generalized Kadomtsev-Petviashvili equation}, Comm. Math. Phys., 183, 253--266, (1997).

%\bibitem{Shatah} J. Shatah, \textit{Stable standing waves of nonlinear Klein-Gordon equations}. Comm. Math. Phys. A.91, 313--327 (1983).


%\bibitem{Fukuizumi} R. Fukuizumi, \textit{Stability and instability of standing waves for the nonlinear Schr\"odinger equation with harmonic potential.} Discrete Contin. Dynam. Syst.,  7, 525--544  (2001).



\end{thebibliography}
\end{document}